\documentclass[a4paper, 11pt]{article}
\usepackage{latexsym}
\usepackage{amsmath}
\usepackage{amsthm}
\usepackage{amssymb}
\usepackage[all]{xy}

\newtheorem{thm}[equation]{Theorem}
\newtheorem{prop}[equation]{Proposition}
\newtheorem{lem}[equation]{Lemma}
\newtheorem{cor}[equation]{Corollary}

\theoremstyle{definition}

\theoremstyle{remark}
\newtheorem{rem}[equation]{Remark}

\newcommand{\N}{\mathbb{N}}

\newcommand{\II}{\mathrm{II}} \newcommand{\IM}{\mathrm{IM}}
\newcommand{\IH}{\mathrm{IH}} \newcommand{\MM}{\mathrm{MM}}
\newcommand{\MH}{\mathrm{MH}} \newcommand{\HH}{\mathrm{HH}}
\newcommand{\MI}{\mathrm{MI}} \newcommand{\HI}{\mathrm{HI}}
\newcommand{\HM}{\mathrm{HM}} 

\newcommand{\C}{C} 
\newcommand{\CII}{\textrm{$\C$-$\II$}} 
\newcommand{\CIH}{\textrm{$\C$-$\IH$}} 
\newcommand{\CMH}{\textrm{$\C$-$\MH$}} \newcommand{\CHH}{\textrm{$\C$-$\HH$}}
\newcommand{\CMI}{\textrm{$\C$-$\MI$}} \newcommand{\CHI}{\textrm{$\C$-$\HI$}}

\newcommand{\XY}{\mathrm{XY}}
\newcommand{\CXY}{\textrm{$\C$-$\XY$}}
\newcommand{\XH}{\mathrm{XH}}
\newcommand{\XI}{\mathrm{XI}}
\newcommand{\HY}{\mathrm{HY}}
\newcommand{\MY}{\mathrm{MY}}
\newcommand{\IY}{\mathrm{IY}}
\newcommand{\CHY}{\textrm{$\C$-$\HY$}}
\newcommand{\CMY}{\textrm{$\C$-$\MY$}}
\newcommand{\CIY}{\textrm{$\C$-$\IY$}}

\newcommand{\sk}{\vspace{\baselineskip}}

\newcommand{\diam}{\mathrm{diam}}
\newcommand{\dist}{\mathrm{dist}}

\newcommand{\ts}{\ensuremath{\Box} \hspace{-0.9mm} \ensuremath{\Box}}
\newcommand{\PCM}{\mathrm{PCM}}

\begin{document}

\author{Deborah C.\ Lockett\\
\footnotesize{School of Mathematics, University of Leeds, Leeds, LS2 9JT, UK}\\
\footnotesize{d.c.lockett@leeds.ac.uk}}
\title{Connected-homomorphism-homogeneous graphs}
\date{}
\maketitle

\begin{abstract}
A relational structure is (connected-)homogeneous if every isomorphism between finite (connected) substructures extends to an automorphism of the structure. We investigate notions which generalise (connected-)homogeneity, where ``isomorphism'' may be replaced by ``homomorphism'' or ``monomorphism'' in the definition. 
Specifically, we study the classes of finite connected-homomorphism-homogeneous graphs, with the aim of producing classifications. The main result is a classification of the finite $\CHH$ graphs, where a graph $G$ is $\CHH$ if every homomorphism from a finite connected induced subgraph of $G$ into $G$ extends to an endomorphism of $G$. The finite $\CII$ (connected-homogeneous) graphs were classified by Gardiner in 1976, and from this we obtain classifications of the finite $\CHI$ and $\CMI$ finite graphs. Although not all the classes of finite connected-homomorphism-homogeneous graphs are completely characterised, we may still obtain the final hierarchy picture for these classes. 
\end{abstract}


\section{Introduction} \label{secintro}

The purpose of this paper is to study certain generalisations of homogeneity, with the aim of obtaining classifications for particular types of relational structures. 
The notion of homogeneity of structures was first defined by Fra\"iss\'e in the 1950s \cite{fr}, and since then the study of these structures has been popular for model theorists, group theorists, combinatorialists, and others. A relational structure is \emph{homogeneous} if every isomorphism between finite substructures extends to an automorphism of the whole structure. Several classification results have been obtained for homogeneous relational structures of different types: the finite homogeneous graphs were classified by Gardiner \cite{gar}; countable graphs by Lachlan and Woodrow \cite{lachlanwoodrow80}; countable posets by Schmerl \cite{schmerl}; countable tournaments by Lachlan \cite{lachlan84}; and countable digraphs by Cherlin \cite{cherlin98}. 

The idea of generalising this notion by replacing isomorphism by monomorphism or homomorphism in the definition was first introduced by Cameron and Ne\v{s}et\v{r}il in 2004 \cite{cn}. 
A \emph{homomorphism} between relational structures of the same type is a map between their base sets which preserves relations; a \emph{monomorphism} is an injective homomorphism; and an \emph{isomorphism} is a bijective homomorphism whose inverse is also a homomorphism. 
A number of notions of \emph{homomorphism-homogeneity} arise --- for instance, we say that a relational structure $S$ is \emph{$\MH$} if every monomorphism from a finite substructure of $S$ into $S$ extends to a homomorphism from $S$ into $S$. Similarly, we may define $\IH, \IM, \II, \MM, \MI, \HH, \HM, \HI$ (note that $\II$ corresponds to the classical notion of homogeneity), and these notions form a natural hierarchy inherited from that of the relation-preserving maps (see Figure~\ref{infhierpic}). 
Note that we use `substructure' in the model theoretic sense, that is, we always mean induced substructures.

Where the classical notion of homogeneity says that every local symmetry is also a global symmetry, here we now sometimes allow our `symmetries' to be weaker --- all relations must still be preserved but non-relations do not have to be, so we allow some `collapsing' of the structure. 

\begin{figure}[h!tb]				
\hspace{3.9cm}
\xymatrix @=2.5pc @dr	{
\IH \ar@{-}[r] \ar@{-}[d] & \MH \ar@{-}[r] \ar@{-}[d] & \HH \ar@{-}[d] \\
\IM \ar@{-}[r] \ar@{-}[d] & \MM \ar@{-}[r] \ar@{-}[d] & \HM \ar@{-}[d] \\
\II \ar@{-}[r] & \MI \ar@{-}[r] & \HI 		}
\caption{Hierarchy picture of the homomorphism-homogeneity classes for countable structures.} \label{infhierpic}
\end{figure}
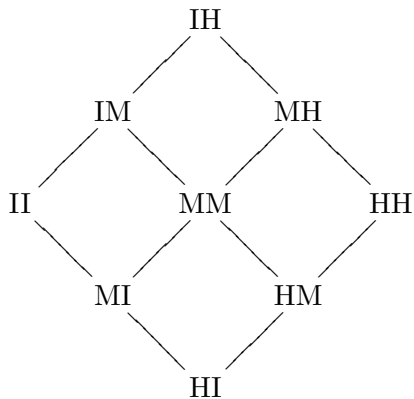

Since Cameron and Ne\v{s}et\v{r}il's paper \cite{cn} on homomorphism-homogeneity a number of authors have published work on the topic.
Classifications of the classes of countable homomorphism-homogeneous posets were completed by Cameron and Lockett \cite{cl:phh}, considering both strict and nonstrict order preserving homomorphisms. The strict order case for $\HH$ posets was also carried out by Ma\v{s}ulovi\'c \cite{masulovic:p}. With various coauthors, Ma\v{s}ulovi\'c has also produced a number of other papers on the classification of $\HH$ binary relational structures (including finite tournaments \cite{masulovic:t}, lattices \cite{masulovic:lat}, and more general structures \cite{masulovic:irreflexbrs,masulovic:reflexbrs}). The motivations for Ma\v{s}ulovi\'c's investigations come from clone theory \cite{masulovic:clone}, and in fact some of his work on the topic of homomorphism-homogeneity preempted the formal definition by Cameron and Ne\v{s}et\v{r}il. In \cite{russchw:hhgraphs}, Rusinov and Schweitzer investigate countable homomorphism-homogeneous graphs (in particular those that are $\MM, \MH, \HH$), obtaining nice results answering questions posed in \cite{cn}. 

In this paper, we continue work begun in \cite{dcl:thesis} studying the classes of finite homomorphism-homogeneous graphs. In this case, the only interesting class (that is, the only class not a subclass of $\II$) is $\IH$, but the classification of such graphs currently remains incomplete.  The purpose of the present paper is to consider a familiar further weakening of homoegeneity in order to find some new meaningful classifications. Thus we restrict to initial maps between connected subgraphs --- a graph is \emph{connected-homogeneous} (or \emph{$\C$-homogeneous}) if every isomorphism between finite connected induced subgraphs extends to an automorphism of the graph. Similarly we define the notions of \emph{connected-homomorphism-homogeneity}. Note that for a finite structure $S$, any monomorphism from $S$ to $S$ must in fact be an isomorphism; thus $\IM$ is the same as $\II$, and so on. The relevant hierarchy of notions that we now consider is shown in Figure~\ref{cfinhierpic}.

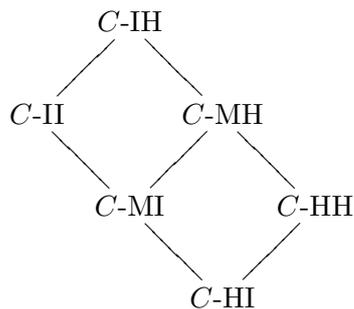
\begin{figure}[h!tb]				
\hspace{4.5cm}
\xymatrix @=2.5pc @dr	{
\CIH \ar@{-}[r] \ar@{-}[d] & \CMH \ar@{-}[r] \ar@{-}[d] & \CHH \ar@{-}[d] \\
\CII \ar@{-}[r] & \CMI \ar@{-}[r] & \CHI 		}
\caption{Hierarchy picture of the connected-homomorphism-homogeneity classes, for finite structures.} \label{cfinhierpic}
\end{figure}

Our starting point for investigating these notions is the classification of the finite $\C$-homogeneous ($\CII$) graphs by Gardiner: 

\begin{thm}[Gardiner \cite{gar78}] \label{CIIfg} 		
A finite graph is $\C$-homogeneous if and only if it is isomorphic to a finite disjoint union of copies of one of the following:
\begin{enumerate}
\item[(i)] a complete graph $K_n$ $(n \ge 1)$;
\item[(ii)] a regular complete t-partite graph $K_t [\overline{K_s}]$ $(s,t \ge 2)$;
\item[(iii)] a cycle $C_n$ $(n \ge 5)$;
\item[(iv)] the line graph of a complete bipartite graph $L(K_{s,s})$ $(s \ge 3)$;
\item[(v)] a bipartite complement of a perfect matching $\overline{L(K_{2,n})}$ $(n \ge 3)$;
\item[(vi)] the Petersen graph $\overline{L(K_5)}$;
\item[(vii)] the Clebsch graph $\square_5$.
\end{enumerate}
\end{thm}

Most of these graphs are well known, but let us briefly describe those that may not be. 
A ``bipartite complement of a perfect matching'' is a bipartite graph with parts $X = \{ x_1, \ldots, x_n \}$, $Y = \{ y_1, \ldots, y_n \}$, such that $x_i \sim y_j$ if and only if $i \ne j$. Such a graph can alternatively be constructed as $\overline{L(K_{2,n})}$, and we shall use this notation. 
The ``line graph of a complete bipartite graph'' $L(K_{s,s})$ has $s^2$ vertices $\{ a_1, \ldots, a_s, b_1, \ldots, b_s, \ldots, z_1, \ldots, z_s \}$ where $| \{ a, b, \ldots, z \} | = s$, such that $u_i \sim v_j$ if and only if $u=v$ or $i=j$, where $u,v \in \{a, \ldots, z \}, ~i,j \in \{ 1, \ldots, s \}$. 
The Clebsch graph $\square_5$ is obtained by identifying antipodal vertices of the 5-dimensional cube $Q_5$ (the induced subgraph on the neighbours of any vertex is a copy of $\overline{K_5}$, while the non-neighbours form the Petersen graph). 

The problem of classifying other connected-homogeneous relational structures has been of recent interest. The countable connected-homogeneous graphs were classified by Gray and Macpherson \cite{graymac}; and the locally-finite connected-homogeneous digraphs were classified by Hamann \cite{hamann}, with certain major subclasses classified by Gray and Moller \cite{graymoller}, and Hamann and Hundertmark \cite{hamannhund}. 

The main result of this paper is the classification of the finite $\CHH$ graphs. Before stating this result, we define some other relevant kinds of graph that appear. 
We first introduce a family of ``treelike'' graphs. 
For $n \ge 2$, a \emph{$K_n$-treelike} finite connected graph $G$ is constructed from copies of $K_n$ (which we call the \emph{components}) by joining some pairs of distinct components $U,V$ by identifying a unique pair of vertices $u \in U$ and $v \in V$. We do this in such a way that we do not construct any new cycles --- that is, if we have an induced cycle in $G$, then it must be contained in a single component. See Figures~\ref{treelikegpic1}, \ref{treelikegpic2} for some examples.

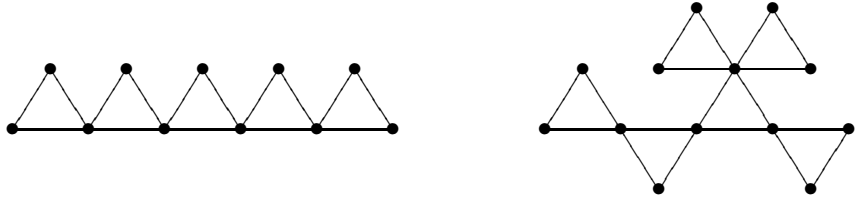
\begin{figure}[h!tb]		
\hspace{2cm}
\begin{xy}
(0,10)*={\bullet}="a" ,
(10,10)*={\bullet}="b" ,
(20,10)*={\bullet}="c" ,
(30,10)*={\bullet}="d" ,
(40,10)*={\bullet}="e" ,
(50,10)*={\bullet}="f" ,
(5,18)*={\bullet}="a1" ,
(15,18)*={\bullet}="b1" ,
(25,18)*={\bullet}="c1" ,
(35,18)*={\bullet}="d1" ,
(45,18)*={\bullet}="e1" ,
"a";"f" **@{-} ,
"a";"a1" **@{-} ,
"b";"b1" **@{-} ,
"c";"c1" **@{-} ,
"d";"d1" **@{-} ,
"e";"e1" **@{-} ,
"b";"a1" **@{-} ,
"c";"b1" **@{-} ,
"d";"c1" **@{-} ,
"e";"d1" **@{-} ,
"f";"e1" **@{-} ,
(70,10)*={\bullet}="i" ,
(80,10)*={\bullet}="j" ,
(90,10)*={\bullet}="k" ,
(100,10)*={\bullet}="l" ,
(110,10)*={\bullet}="m" ,
(75,18)*={\bullet}="i1" ,
(85,2)*={\bullet}="j1" ,
(95,18)*={\bullet}="k1" ,
(90,26)*={\bullet}="k2" ,
(85,18)*={\bullet}="k3" ,
(100,26)*={\bullet}="k22" ,
(105,18)*={\bullet}="k23" ,
(105,2)*={\bullet}="l1" ,
"i";"m" **@{-} ,
"i";"i1" **@{-} ,
"i1";"j1" **@{-} ,
"j1";"k22" **@{-} ,
"l1";"k2" **@{-} ,
"k3";"k23" **@{-} ,
"k2";"k3" **@{-} ,
"k22";"k23" **@{-} ,
"l1";"m" **@{-} ,
\end{xy}
\caption{Some examples of $K_3$-treelike graphs.} \label{treelikegpic1}
\end{figure}

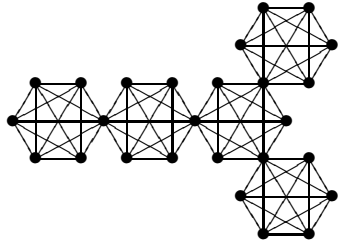
\begin{figure}[h!tb]		
\hspace{5cm}
\begin{xy}
(0,20)*={\bullet}="a1" ,
(3,25)*={\bullet}="a2" ,
(9,25)*={\bullet}="a3" ,
(12,20)*={\bullet}="ab" ,
(9,15)*={\bullet}="a5" ,
(3,15)*={\bullet}="a6" ,
(15,25)*={\bullet}="b2" ,
(21,25)*={\bullet}="b3" ,
(24,20)*={\bullet}="bc" ,
(21,15)*={\bullet}="b5" ,
(15,15)*={\bullet}="b6" ,
(27,25)*={\bullet}="c2" ,
(33,25)*={\bullet}="cd" ,
(36,20)*={\bullet}="c4" ,
(33,15)*={\bullet}="ce" ,
(27,15)*={\bullet}="c6" ,
(30,30)*={\bullet}="d1" ,
(33,35)*={\bullet}="d2" ,
(39,35)*={\bullet}="d3" ,
(42,30)*={\bullet}="d4" ,
(39,25)*={\bullet}="d5" ,
(30,10)*={\bullet}="e1" ,
(39,15)*={\bullet}="e3" ,
(42,10)*={\bullet}="e4" ,
(39,5)*={\bullet}="e5" ,
(33,5)*={\bullet}="e6" ,
"a1";"c4" **@{-} ,
"a1";"a2" **@{-} ,
"a1";"a3" **@{-} ,
"a1";"a5" **@{-} ,
"a1";"a6" **@{-} ,
"ab";"a2" **@{-} ,
"ab";"a3" **@{-} ,
"ab";"a5" **@{-} ,
"ab";"a6" **@{-} ,
"a2";"a3" **@{-} ,
"a2";"a5" **@{-} ,
"a2";"a6" **@{-} ,
"a3";"a5" **@{-} ,
"a3";"a6" **@{-} ,
"a5";"a6" **@{-} ,
"ab";"b2" **@{-} ,
"ab";"b3" **@{-} ,
"ab";"b5" **@{-} ,
"ab";"b6" **@{-} ,
"bc";"b2" **@{-} ,
"bc";"b3" **@{-} ,
"bc";"b5" **@{-} ,
"bc";"b6" **@{-} ,
"b2";"b3" **@{-} ,
"b2";"b5" **@{-} ,
"b2";"b6" **@{-} ,
"b3";"b5" **@{-} ,
"b3";"b6" **@{-} ,
"b5";"b6" **@{-} ,
"bc";"c2" **@{-} ,
"bc";"cd" **@{-} ,
"bc";"ce" **@{-} ,
"bc";"c6" **@{-} ,
"c4";"c2" **@{-} ,
"c4";"cd" **@{-} ,
"c4";"ce" **@{-} ,
"c4";"c6" **@{-} ,
"c2";"cd" **@{-} ,
"c2";"ce" **@{-} ,
"c2";"c6" **@{-} ,
"cd";"ce" **@{-} ,
"cd";"c6" **@{-} ,
"ce";"c6" **@{-} ,
"d1";"d4" **@{-} ,
"d1";"d2" **@{-} ,
"d1";"d3" **@{-} ,
"d1";"d5" **@{-} ,
"d1";"cd" **@{-} ,
"d4";"d2" **@{-} ,
"d4";"d3" **@{-} ,
"d4";"d5" **@{-} ,
"d4";"cd" **@{-} ,
"d2";"d3" **@{-} ,
"d2";"d5" **@{-} ,
"d2";"cd" **@{-} ,
"d3";"d5" **@{-} ,
"d3";"cd" **@{-} ,
"d5";"cd" **@{-} ,
"e1";"e4" **@{-} ,
"e1";"ce" **@{-} ,
"e1";"e3" **@{-} ,
"e1";"e5" **@{-} ,
"e1";"e6" **@{-} ,
"e4";"ce" **@{-} ,
"e4";"e3" **@{-} ,
"e4";"e5" **@{-} ,
"e4";"e6" **@{-} ,
"ce";"e3" **@{-} ,
"ce";"e5" **@{-} ,
"ce";"e6" **@{-} ,
"e3";"e5" **@{-} ,
"e3";"e6" **@{-} ,
"e5";"e6" **@{-} 
\end{xy}
\caption{An example of a $K_6$-treelike graph.} \label{treelikegpic2}
\end{figure}

Let us make some remarks about these graphs. Firstly, complete graphs are trivial $K_n$-treelike graphs which have just one component. 
Secondly, a tree is a $K_2$-treelike graph. 
Also, just as trees are characterised as being the connected graphs with no induced cycles, so we may characterise the $K_n$- treelike graphs as the following: 
the connected graphs such that the only induced cycles are triangles, and the neighbour set of each vertex is a disjoint union of $K_{n-1}$ graphs. 
Finally, we remark that these $K_n$-treelike graphs are finite generalisations of the treelike infinite locally-finite distance transitive graphs introduced by Macpherson in~\cite{macph}, one of the families of countable $\C$-homogeneous graphs \cite{graymac}. These are constructed from semi-regular trees, and are essentially regular infinite versions of our $K_n$-treelike graphs. 

Next we introduce a special family of connected bipartite graphs related to the bipartite complement of perfect matching graphs. 
We say that a bipartite graph with parts $X,Y$, such that $|X| \le |Y|$, has a \emph{perfect complement matching} if for each $x \in X$, there is a vertex $y_x \in Y$ such that $x \nsim y_x$, and for $x \ne x'$ we have $y_x \ne y_{x'}$. 
If $G$ is a finite connected bipartite graph with parts $X,Y$ such that $2 \le |X| \le |Y| = n$ and $G$ has a perfect complement matching, then we say that $G$ is a \emph{$\PCM(n)$ graph}. 
So a $\PCM(n)$ graph is a connected subgraph of $\overline{L(K_{2,n})}$ which spans the whole of one part (so that $|Y| = n$) --- note here we really mean subgraph in the usual graph-theoretic sense, that is, we may only have a subset of the edges. 
Clearly $\overline{L(K_{2,n})}$ is a special case of a $\PCM(n)$ graph --- with both parts of size $n$, which is complete bipartite except for a unique perfect complement matching. If $G$ does not embed a $\PCM(n)$ graph then we say that $G$ is \emph{$\PCM(n)$-free}  

There is one other particular finite graph which plays an important role in the classification. This is the \emph{two-squares} graph (a 6-cycle with one diagonal), which we denote by $\ts$ (following~\cite{graymac}), see Figure~\ref{tspic} for a picture.

\begin{figure}[h!tb]		
\hspace{6cm}
\begin{xy}
(0,0)*={\bullet}="a" ,
(0,10)*={\bullet}="b" ,
(10,10)*={\bullet}="c" ,
(20,10)*={\bullet}="d" ,
(20,0)*={\bullet}="e" ,
(10,0)*={\bullet}="f" ,
"a";"b" **@{-} ,
"b";"c" **@{-} ,
"c";"d" **@{-} ,
"d";"e" **@{-} ,
"e";"f" **@{-} ,
"f";"a" **@{-} ,
"c";"f" **@{-} ,
\end{xy}
\caption{The two-squares graph $\ts$.} \label{tspic}
\end{figure}
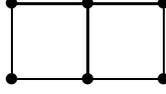

We may now state our main classification results. We first produce a classification of the finite $\CHH$ graphs in the connected case: 

\begin{thm} \label{CHHcfg}				
Let $G$ be a finite connected graph. Then $G$ is $\CHH$ if and only if it is one of the following: 
\begin{enumerate}
\item[(i)] $K_1$;
\item[(ii)] a $K_n$-treelike graph $(n \ge 2)$;
\item[(iii)] a graph such that all induced cycles are squares, but $\ts$ does not embed;
\item[(iv)] a bipartite graph such that each part has a common neighbour; 
\item[(v)] the bipartite complement of a perfect matching $\overline{L(K_{2,n})}$ $(n \ge 3)$.
\end{enumerate}
\end{thm}

Note that like families (iv) and (v), all graphs in family (iii) are bipartite. 
Observe that trees are $K_2$-treelike graphs, and since they have no induced cycles they also vacuously satisfy the property that all induced cycles are squares but $\ts$ does not embed. So the family of all finite trees is the intersection of families (ii) and (iii) (since for $n \ge 3$, $K_n$-treelike graphs are not bipartite). 
We may also note that a complete bipartite graph $K_{m,n}$ is both a graph such that all induced cycles are squares but $\ts$ does not embed, and a bipartite graph such that each part has a common neighbour; so the family of complete bipartite graphs lies in the intersection of families (iii) and (iv).

Next we characterise how the disconnected cases may be constructed as certain unions of these:

\begin{thm} \label{CHHfg}				
A finite graph $G$ is $\CHH$ if and only if it is a finite disjoint union of finite connected $\CHH$ graphs $\underset{i \in [k]} \bigcup G_i$ such that one of the following holds: 
\begin{enumerate}
\item[(a)] $G$ is an independent set;
\item[(b)] each $G_i$ is a $K_n$-treelike graph, for fixed $n \ge 3$;
\item[(c)] each $G_i$ is a graph such that all induced cycles are squares, but $\ts$ does not embed;
\item[(d)] each $G_i$ is a bipartite graph such that each part has a common neighbour;
\item[(e)] for fixed $n \ge 3$, some of the components are copies of $\overline{L(K_{2,n})}$, and all other components $G_i$ are bipartite $\PCM(n)$-free graphs such that each part has a common neighbour.
\end{enumerate}
\end{thm}

Furthermore, we produce classifications of the finite $\CHI$ and $\CMI$ graphs, and also establish that the hierarchy picture for the classes of finite connected-homomorphism-homogeneous graphs does not reduce in any way, since there are no inclusions other than those shown in Figure~\ref{cfinhierpic}. 

The rest of the paper is organised as follows: in Section~\ref{secprelim} we formally describe the relevant notation and terminology relating to graphs, graph homomorphisms, and homomorphism-homogeneity of general relational structures, including the introduction of the `($\C$-)homomorphism-homogeneous correspondences', and give some preliminary results. We then focus on the case of finite graphs, in Section~\ref{secChomhom} we develop the theory relating to the ($\C$-)homomorphism-homogenous correspondences. The remainder of the paper is concerned with producing classifications of the classes of finite $\C$-homomorphism-homogeneous graphs: Section~\ref{secCHICMI} covers those that are $\CHI$ and $\CMI$; the next two sections cover those that are $\CHH$, in Sections~\ref{secCHHcfg}, \ref{secCHHfg} we prove Theorems~\ref{CHHcfg}, \ref{CHHfg} respectively. Finally, the paper concludes in Section~\ref{secCIHCMH} with a discussion of the remaining classes $\CIH$ and $\CMH$, and open questions.

\section{Preliminaries} \label{secprelim}

\subsection{Graph Theory}

We begin with an overview of the graph-theoretic notions and notation that will be used in this paper, most of which is standard, see for instance \cite{diestel}. 

A \emph{graph} $G$ is a pair $(V(G),E(G))$ where $V(G)$ is a non-empty set of points (called the \emph{vertices}) and $E(G) \subseteq V(G)^{\{2 \}}$ (pairs in $E(G)$ are called \emph{edges}). 
All graphs will be \emph{simple} (no loops or multiple edges). 
We will tend to simplify and refer to $v \in G$ rather than $v \in V(G)$, and more generally $U \subseteq G$ rather than $U \subseteq V(G)$. 
Recall that we deal with substructures in the model theoretic sense, so (unless explicitly stated) the term \emph{subgraph} is used to refer to what is more usually called an \emph{induced subgraph} in graph theory. 
Strictly, we write $\langle A \rangle$ for the (induced) subgraph of $G$ with vertex set $A \subset G$, but we may also be relaxed about this and refer to $A$ itself as the subgraph. 
If $H$ is isomorphic to an induced subgraph of $G$, then we say that $H$ \emph{embeds} in $G$.

We write $u \sim v$ and say $u,v$ are \emph{adjacent} if $u, v \in G$ are joined by an edge.
We write $v \sim A$, and say $v \in G$ is a \emph{common neighbour} of $A \subset G $, if $v \sim a$ for each $a \in A$. 
Similarly we write $v \nsim A$, if $v \nsim a$ for each $a \in A$. 
Let $N(v):= \{u \in G : u \sim v \}$ be the \emph{neighbour set} of $v \in G$, and $N(V):= \{u \in G : u \sim V\} = \underset{v \in V}\bigcap N(v)$ be the set of common neighbours of $V \subset G$.
Let $d(v) := |N(v)|$ be the \emph{degree} of $v \in G$, and $\Delta(G) := \max \{ d(v) : v \in G \}$ be the \emph{maximum degree} of $G$.

A \emph{path} is a graph $P = (V,E)$ of the form $V= \{ x_0, x_1, x_2, \ldots, x_k \}$, $E = \{ x_0x_1, x_1x_2, \ldots, x_{k-1}x_k \}$. 
The \emph{length} of a path is the number of edges it contains, and a path of length $k$ is called a \emph{$k$-path}, and denoted by $P_k$. 
We may write this path as $P = x_0 x_1 \ldots x_k$, and call $P$ a path \emph{between} its \emph{ends} $x_0$ and $x_k$. 

If $P = x_0 x_1 \ldots x_{k-1}$ is a path and $k \ge 3$, then the graph $C:= P + x_{k-1}x_0$ is called a \emph{cycle}. We may write this as $C = x_0 x_1 \ldots x_{k-1} x_0$. As for paths, the \emph{length} of a cycle is the number of edges it contains, and a cycle of length $k$ is called a \emph{$k$-cycle}, and denoted by $C_k$. 

An \emph{induced path (cycle)} in $G$ is a path (cycle) in $G$ forming an induced subgraph. So an induced cycle is one that has no chords, and similarly an induced path is one for which there are no additional edges. 
We call a $3$-cycle a \emph{triangle}, and an induced $4$-cycle a \emph{square}.
The minimum length of a cycle in $G$ is the \emph{girth} of $G$, denoted by $g(G)$.

A graph is \emph{connected} if there is a path between any two of its vertices. 
In a connected graph $G$, the \emph{distance} $d(x,y)$ between two vertices $x,y$ is the minimum length of a path with end vertices $x$ and $y$. 
The greatest distance between any two vertices in $G$ is the \emph{diameter} of $G$, denoted by $\diam(G)$. 
If a graph is not connected, then it is \emph{disconnected}, and we may refer to its \emph{connected components} (maximal connected induced subgraphs).

A \emph{tree} is a connected graph with no cycles.

\subsection{Graph homomorphisms}

A graph homomorphism is a map which preserves edges; non-edges may get mapped to edges, non-edges, or a single vertex. There is a rich and interesting body of research relating to graph homomorphisms, see for instance \cite{hen}. Graphs $G_1, G_2$ are \emph{homomorphically equivalent} if there is a homomorphism mapping $G_1$ into $G_2$, and vice versa. The (unique) smallest graph in any homomorphic-equivalence class of finite graphs is called a \emph{core}; and if $G, C$ are homomorphically equivalent and $C$ is a core, then we call $C$ the \emph{core of $G$}.
For example, it is straightforward to see that the family of finite bipartite graphs is a homomorphic-equivalence class, with core $K_2$.
Core graphs are also characterised by the property that every endomorphism of a core is actually an automorphism.

Bipartite graphs play an important role in the $\C$-homomorphism-homogeneous classes, and we make much use of the following simple result about homomorphisms within bipartite graphs. 

\begin{lem} \label{biparthom} 			
If $\phi$ is a homomorphism between connected bipartite graphs, then $\phi$ preserves the bipartitions. 
In particular, if $G$ is a bipartite graph, and $\phi$ is a homomorphism between connected subgraphs of $G$, then $\phi$ preserves the bipartition of $G$.
\end{lem}

\begin{proof}
First observe that for any pair of vertices $u,v$ in a connected bipartite graph, we may determine whether $u,v$ are in the same part or different parts of the bipartition by the length of the paths from $u$ to $v$. Note that all paths from $u$ to $v$ must have the same parity (since there are no odd cycles in a bipartite graph); then $u,v$ are in the same part if and only if the paths all have even length, while $u,v$ are in different parts if and only if the paths all have odd length. 

Next notice that if $f: P \to P'$ is a homomorphism from one path to another, then since $f$ preserves edges, the parity of the distance between the ends of $P$ is preserved. 

So let $\phi$ be a homomorphism from the connected subgraph $A$ into $G$.
Consider $u,v \in A$, with $u \ne v$. If $u,v$ are in the same part, then they are an even distance apart; so then $\phi(u), \phi(v)$ will also be an even distance apart, and hence $\phi(u), \phi(v)$ are also in the same part. Otherwise, if $u,v$ are in different parts, then they are an odd distance apart; so then $\phi(u), \phi(v)$ will also be an odd distance apart, and hence $\phi(u), \phi(v)$ are also in different parts. 
Thus $\phi$ preserves the partition.
\end{proof}

\subsection{Homomorphism-homogeneity}

This section collects together some relevant preliminary results about homomorphism-homogeneity. First of all, recall that Cameron and Ne\v{s}et\v{r}il showed that the classes $\MH$ and $\HH$ of finite graphs coincide, and their classification is rather trivial. 

\begin{thm}[Cameron, Ne\v set\v ril \cite{cn}] \label{MH=HHfg} 	
A finite graph is $\MH$ or $\HH$ if and only if it is a disjoint union of complete graphs of the same size.
\end{thm}

This paper is almost exclusively concerned with the specific case of connected-homomorphism-homogeneity for finite graphs, however we develop the theory in the general relational structure context (where structures may be infinite).
Generalising the notion of connectivity for graphs, a relational structure $S$ is \emph{connected} if for each pair of points $x,y \in S$ there is a sequence $x=x_0, x_1, \ldots, x_{k-1}, x_k=y \in S$ such that for each $i \in [k] := \{ 1,2, \ldots, k \}$, points $x_{i-1}, x_i$ are related by some (possibly higher arity) relation of $S$. As for graphs, we call this sequence a \emph{path}.

Recall the definition of the notions of (connected-)homomorphism-homogeneity: for $(\mathrm{X},x), (\mathrm{Y},y) \in \{ \mathrm{(I,iso), (M,mono), (H,homo)} \}$, a relational 
structure $S$ is \emph{$(\C$-$)\XY$} if every $x$-morphism from a finite (connected) substructure of $S$ into $S$ extends to a $y$-morphism from $S$ to $S$.

As classes of structures, clearly $\XY \subseteq \CXY$, so note that Theorem~\ref{MH=HHfg} provides a direct starting point (although a rather weak one) for the corresponding classifications of finite $\CMH$ and $\CHH$ graphs. In the next section we will also see how $\XY$ graphs appear as induced subgraphs on neighbour sets in $\CXY$ graphs, which is more useful.

We now extend our notions further. Let $S_1, S_2$ be two connected $(\C$-$)\XY$ relational structures of the same type. Then we say that $S_1$ is \emph{$(\C$-$)\XY$-morphic} to $S_2$ if every $x$-morphism from a finite (connected) substructure $A \subseteq S_1$ onto a substructure $B \subseteq S_2$ extends to a $y$-morphism from $S_1$ to $S_2$. We say that $S_1, S_2$ are \emph{$(\C$-$)\XY$-symmetric} if $S_1$ is $(\C$-$)\XY$-morphic to $S_2$ and $S_2$ is $(\C$-$)\XY$-morphic to $S_1$.
We refer to these notions as the \emph{($\C$-)homomorphism-homogeneous} (or \emph{($\C$-)hom-hom) correspondences}.

\begin{prop} \label{discXYrs} 			
If $S$ is a relational structure, then $S$ is $\CXY$ if and only if all connected components of $S$ are $\CXY$, and they are all pairwise $\CXY$-symmetric.
\end{prop}

\begin{proof}
First suppose that $S$ is $\CXY$, and consider any pair of connected components $S_i, S_j$ of $S$.
Let $\phi: A \to B$ be a $x$-morphism from the finite connected substructure $A \subseteq S_i$ onto $B \subseteq S_j$. 
Since $S$ is $\CXY$, this $x$-morphism between finite connected substructures of $S$ extends to a $y$-morphism $\psi: S \to S$.

We claim that $\psi \big|_{S_i}$ is a $y$-morphism from $S_i$ into $S_j$.
Let $a \in A \subseteq S_i$.
If $x \in S_i$, then there is a path from $a$ to $x$.
Now since $\psi$ is relation-preserving, there is a path from $\psi(a) \in S_j$ to $\psi(x)$, and so $\psi(x) \in S_j$.
Thus $\psi \big|_{S_i}$ maps $S_i$ into $S_j$.
Furthermore, since $\psi$ is a $y$-morphism, $\psi \big|_{S^*}$ is also a $y$-morphism for any $S^* \subseteq S$.

So, if $i=j$, then this tells us that $S_i$ is $\CXY$; 
and if $i \ne j$, then this tells us that $S_i$ and $S_j$ are $\CXY$-symmetric. 

Conversely, let $S =  \underset{i \in I}\bigcup S_i$ where each $S_i$ is a connected component, and suppose that each $S_i$ is $\CXY$, and all pairs of components are $\CXY$-symmetric. 
Consider any $x$-morphism $\phi$ from finite connected $A \subset S$ onto $B \subset S$.
Now since $A$ is connected, it must be contained within a single connected component, and similarly for $B$. So $A \subseteq S_i, ~B \subseteq S_j$ for some $i,j \in I$.

If $i=j$, then since $S_i$ is $\CXY$ for every $i \in I$, there is a $y$-morphism $\psi_i: S_i \to S_i$ that extends $\phi$.
Now define the map $\psi: S \to S$ by
\[
\psi(v) = \left\{ \begin{array}{ll}
\psi_i (v) & \textrm{if $v \in S_i$}\\
v & \textrm{if $v \notin S_i$}.
\end{array} \right.
\]
Then $\psi$ clearly extends $\phi$, and is a $y$-morphism from $S$ to $S$ because it is a $y$-morphism on $S_i$ and an isomorphism on $S \setminus S_i$.

If $i \ne j$, then since $S_i, S_j$ are $\CXY$-symmetric, there is a $y$-morphism $\psi_i: S_i \to S_j$ that extends $\phi$, and also there is certainly some $y$-morphism $\psi_j: S_j \to S_i$.
Now define the map $\psi: S \to S$ by
\[
\psi(x) = \left\{ \begin{array}{ll}
\psi_i (x) & \textrm{if $x \in S_i$}\\
\psi_j (x) & \textrm{if $x \in S_j$}\\
x & \textrm{if $x \in S \setminus (S_i \cup S_j)$}.
\end{array} \right.
\]
Then $\psi$ clearly extends $\phi$, and is a $y$-morphism from $S$ to $S$ because it is a $y$-morphism on $S_i \cup S_j$ and an isomorphism on $S \setminus (S_i \cup S_j)$.

Thus in either case we have found a $y$-morphism $\psi: S \to S$ which extends $\phi$, and hence $S$ is $\CXY$.
\end{proof}

\section{$\C$-homomorphism-homogeneous correspondences} \label{secChomhom}

From now on, we focus on graphs, and in particular, finite graphs. In this section we investigate the $\C$-hom-hom correspondences in a general context, keeping in mind that we will later focus on the problem of classifying the finite $\CHH$ graphs. 

First, note that we may make the following observation simply from the definitions:

\begin{rem} \label{symmetric}
If graphs $G_1, G_2$ are $(\C$-$)\XI$-morphic, then by definition they must in fact be isomorphic; and similarly if they are $(\C$-$)\XH$-symmetric, then they must be homomorphically equivalent.
\end{rem}

The next result relates the notions of $\C$-homomorphism-homogeneity and homomorphism-homogeneity to the classical notion of homomorphic equivalence for finite graphs.

\begin{lem} \label{CXHcoreCXI} 			
If graph $G$ is ($\C$-)$\XH$ and $C$ is the core of $G$, then $C$ is ($\C$-)$\XI$, and $C$ and $G$ are ($\C$-)$\XH$-symmetric.
\end{lem}

\begin{proof}
Note that $C$ embeds in $G$, so we may consider $C$ as an induced subgraph of $G$, and since $C$ is the core of $G$ we have a retraction $r: G \to C$ (that is, a surjective homomorphism which is the identity on $C$). 

If $\phi$ is an $x$-morphism between (connected) subgraphs of $C$, then it is also an $x$-morphism between (connected) subgraphs of $G$. So since $G$ is ($\C$-)$\XH$ it can be extended to a homomorphism $\psi: G \to G$. Then the homomorphism $r \circ \psi \big|_C$ is an extension of $\phi$; and it is an automorphism of $C$ because $C$ is a core. So $C$ is ($\C$-)$\XI$.

Now if $\phi$ is an $x$-morphism from a (connected) subgraph of $C$ into $G$, then it is also an $x$-morphism from a (connected) subgraph of $G$ into $G$ since $C$ is a subgraph of $G$. Then since $G$ is ($\C$-)$\XH$, $\phi$ can be extended to a homomorphism $\psi: G \to G$; and $\psi \big|_C : C \to G$ is the required extension of $\phi$. So $C$ is ($\C$-)$\XH$-morphic to $G$.

Conversely, if $\phi$ is an $x$-morphism from a (connected) subgraph of $G$ into $C$ (which is a subgraph of $G$), then since $G$ is ($\C$-)$\XH$ it can be extended to a homomorphism $\psi: G \to G$. Then $r \circ \psi : G \to C$ is a homomorphism which extends $\phi$. So $G$ is ($\C$-)$\XH$-morphic to $C$.
\end{proof}

Now we see that unlike homomorphic equivalence, the $\CHH$-symmetry relation is not an equivalence relation on the class of graphs.

\begin{lem}  \label{CHHsymrel} 		
The $\CHH$-symmetry relation is not transitive on the class of all $\CHH$ graphs.
\end{lem}

\begin{proof}
Consider $K_2$, the 6-cycle $C_6$, and the 4-path $P_4$. 
Let $K_2$ have vertices $\{ u_1, u_2\}$; let $C_6$ have vertices $\{v_1, \ldots, v_6 \}$; and let $P_4$ have vertices $\{ w_1, \ldots, w_5 \}$. 
It is relatively straightforward to see that these graphs are all $\CHH$, that they are all homomorphically equivalent (since they are all bipartite), and that $K_2$ is a core. Then by Lemma~\ref{CXHcoreCXI}, $K_2$ and $C_6$ are $\CHH$-symmetric, and $K_2$ and $P_4$ are $\CHH$-symmetric. 

Meanwhile, $C_6$ and $P_4$ are not $\CHH$-symmetric. In fact, $C_6$ is not even $\CIH$-morphic to $P_4$. 
Consider the isomorphism $\phi: v_i \mapsto w_i$ for $i \in [5]$. 
Now $v_6 \sim \{ v_1, v_5 \}$, but there is no vertex of $P_4$ adjacent to $\{ \phi(v_1), \phi(v_5) \} = \{ w_1, w_5 \}$. So the isomorphism $\phi$ between connected subgraphs of $C_6$ and $P_4$ can not be extended to a homomorphism, and so $C_6$ is not $\CIH$-morphic to $P_4$.
\end{proof}

By Proposition~\ref{discXYrs}, disconnected $\CHH$ graphs are constructed from disjoint unions of $\CHH$ graphs, which are all pairwise $\CHH$-symmetric. 
But now since $\CHH$-symmetry is not an equivalence relation, we can already see that in general we will require care in describing the families of graphs that are pairwise $\CHH$-symmetric.
However, for some particular classes of graphs this description will be straightforward, for instance, we will see that any pair of trees are $\CHH$-symmetric; and thus any forest is $\CHH$.
The problems of classifying finite connected and disconnected $\CHH$ graphs are solved in Sections~\ref{secCHHcfg}, \ref{secCHHfg} respectively.

We end this section with some other general results about $\CXY$ graphs.
In the classification of homogeneous graphs (and analogously for more general relational structures), to prove that a list is complete a common technique is to use an induction argument based on the fact that in a homogeneous graph, the induced subgraph on the neighbour set of a point is also homogeneous (see for instance \cite{gar}, \cite{enomoto}, \cite{cherlin98}, \cite{graymac}). 
Here we look at the analogous results for neighbour sets of subgraphs of $\C$-homomorphism-homogeneous graphs. 
Under certain conditions, for a $\CXY$ graph with $f: U \to V$ an $x$-morphism between subgraphs, not only are $N(U)$ and $N(V)$ $\CXY$-morphic, they are in fact $\XY$-morphic; so in particular $N(V)$ is not just $\CXY$ but is in fact $\XY$. 

For graphs $G, H$ we write $G~\overline{\cup}~H$ to denote the \emph{edge-complete union of $G$ and $H$}. That is, $V(G~\overline{\cup}~H) = V(G) \cup V(H)$, $E(G~\overline{\cup}~H) = E(G) \cup E(H) \cup \{ uv : u \in G, v \in H \}$. So note that for any $V \subset G$, $\langle V \cup N(V) \rangle = V~\overline{\cup}~N(V)$.

\begin{lem} \label{CXYsubneighbours} 			
If $G$ is a $\CXY$ graph and $f$ is an $x$-morphism between finite subgraphs of $G$ from $U$ onto $V$, such that either
\begin{enumerate}
\item[(a)]
$U$ is connected; or 
\item[(b)]
$N(U), N(V) \ne \emptyset$;
\end{enumerate}
then $N(U)$ and $N(V)$ are $\XY$-morphic. 
In particular, for each finite $V \subset G$, $N(V)$ is $\XY$.
\end{lem}

\begin{proof}
Let $\phi: A \to B$ be an $x$-morphism between the finite subgraphs $A \subseteq N(U)$ and $B \subseteq N(V)$, we wish to extend this to a $y$-morphism from $N(U)$ to $N(V)$.

If $N(U) = \emptyset$, then the result is trivial (vacuously the empty graph is $\XY$-morphic to any graph). 

If $N(V) = \emptyset$, then we claim that we must have that $N(U) = \emptyset$; so indeed $N(U)$ and $N(V)$ are $\XY$-morphic.
First note that if $N(V) = \emptyset$, then we must have that $U$ is connected (condition (b) does not hold, so (a) must hold).
So the map $f$ is in fact an $x$-morphism between connected subgraphs of $G$; and since $G$ is $\CXY$, we can extend $f$ to a $y$-morphism $g: G \to G$. 
Now suppose $N(U) \ne \emptyset$, say $u \in N(U)$; but then since $u \sim U$ we must have $g(u) \sim g(U) = f(U) = V$. So then $g(u) \in N(V)$, which contradicts the assumption that $N(V) = \emptyset$. 

So now we may assume that $N(U), N(V) \ne \emptyset$ (so then certainly $U, V \ne \emptyset$), and moreover that $A, B \ne \emptyset$.
Extend $\phi$ to the $x$-morphism $\phi' : A~\overline{\cup}~U \to B~\overline{\cup}~V$, by defining 
\[
\phi'(v) = \left\{ \begin{array}{ll}
\phi (v) & \textrm{if $v \in A$}\\
f(v) & \textrm{if $v \in U$}.
\end{array} \right.
\]
We know that this is indeed an $x$-morphism since $G$ is edge-complete between $U$ and $N(U)$, and between $V$ and $N(V)$; and so in particular $G$ is edge-complete between $U$ and $A$, and between $V$ and $B$.
Furthermore, since $U,V, A,B \ne \emptyset$, we know that $U~\overline{\cup}~A$ and $V~\overline{\cup}~B$ are both connected subgraphs. 
Then $\phi'$ is an $x$-morphism between finite connected subgraphs of $G$.
So since $G$ is $\CXY$, there exists a $y$-morphism $\psi: G \to G$ which extends $\phi'$. 
Then we claim that $\psi \big|_{N(U)}$ is a $y$-morphism from $N(U)$ into $N(V)$. 
Notice that since $\psi$ is edge-preserving, if $x \sim y$ then $\psi(x) \sim \psi(y)$, so for each $z \in G$ we have that if $z \sim U$ then $\psi(z) \sim \psi(U) = V$.
That is, if $z \in N(U)$ then $\psi(z) \in N(V)$; so $\psi(N(U)) \subseteq N(V)$.
Clearly $\psi \big|_{N(U)}$ extends $\phi$, as required.
Thus $N(U)$ and $N(V)$ are $\XY$-morphic.

In particular, if $U = V$, then we have that $\psi \big|_{N(V)}$ is a $y$-morphism from $N(V)$ into $N(V)$ which extends $\phi$, and hence $N(V)$ is $\XY$.
\end{proof}

Note that the conditions (a) or (b) are necessary, since in a general $\CXY$ graph there may be $x$-morphisms between finite disconnected subgraphs where the neighbour set of one subgraph is empty and the other is nonempty, and so certainly the neighbour sets are not $\XY$-morphic (since no nonempty graph is even homomorphic to the empty graph). For instance, there are $\CIH$ graphs with diameter greater than two, and such a graph has non-adjacent pairs of vertices with and without common neighbours. 

Observe that the previous result can be further strengthened if $G$ is $\CHY$ and $U$ and $V$ are in fact homomorphically equivalent, or if $G$ is $\CIY$ or $\CMY$ and $U$ and $V$ are isomorphic.

\begin{cor} \label{CXYsubncor} 			
If $G$ is $\CHY$, the finite subgraphs $U,V$ are homomorphically equivalent, 
and either $U$ is connected or $N(U), N(V) \ne \emptyset$, then $N(U)$ and $N(V)$ are $\HY$-symmetric.
If $G$ is $\CIY$ ($\CMY$), the finite subgraphs $U,V$ are isomorphic, and either $U,V$ are connected or $N(U), N(V) \ne \emptyset$, then $N(U)$ and $N(V)$ are $\IY$-symmetric ($\MY$-symmetric).
\end{cor}

\begin{proof}
Follows directly from Lemma~\ref{CXYsubneighbours}. 
\end{proof}

\section{The classes $\CHI$ and $\CMI$}  \label{secCHICMI}

We now focus on obtaining classifications for the classes of finite $\C$-homomorphism-homogeneous graphs. We start in this section by giving complete classifications for the classes of finite graphs that are $\CHI$ and $\CMI$. Clearly these are subclasses of Gardiner's class of finite $\CII$ graphs (Theorem~\ref{CIIfg}), and so the classification is quite simple. 

Before giving the results, let us first explore the nature of these notions. A graph $G$ is $\CHI$ ($\CMI$) if every homomorphism (monomorphism) from a connected subgraph of $G$ into $G$ extends to an automorphism of $G$. In either case, in order to be extended to an automorphism, the initial map must have been an isomorphism to begin with. So in fact, a graph $G$ is $\CHI$ ($\CMI$) if it is $\CII$ and every homomorphism (monomorphism) from a connected subgraph of $G$ into $G$ is an isomorphism. 

\begin{prop} \label{CHIfg} 			
A countable (finite or infinite) graph is $\CHI$ if and only if it is a disjoint union of copies of a complete graph.
\end{prop}

\begin{proof}
It is easy to see that any complete graph is $\CHI$.
However if $G$ is a connected graph which is not complete, then it has an induced 2-path $xyz$. Then $x \mapsto x,~y \mapsto y,~z \mapsto x$ is a homomorphism from a connected subgraph of $G$ into $G$ which is not an isomorphism, and thus $G$ is not $\CHI$. 
Hence by Proposition~\ref{discXYrs} and Remark~\ref{symmetric} we have the result.
\end{proof}

\begin{thm} \label{CMIfg}		
A finite graph is $\CMI$ if and only if it is isomorphic to a disjoint union of copies of one of the following:
\begin{enumerate}
\item[(i)] a complete graph $K_n$ $(n \ge 1)$;
\item[(ii)] a complete bipartite graph with parts of the same size $K_{s,s}$ $(s \ge 2)$;
\item[(iii)] a cycle $C_n$ $(n \ge 3)$.
\end{enumerate}
\end{thm}

\begin{proof} 
By Proposition~\ref{discXYrs} and Remark~\ref{symmetric}, we just need to show that a finite connected graph is $\CMI$ if and only if it is on the list. 

Firstly, clearly any complete graph is $\CMI$, since these are infact $\CHI$ (Proposition~\ref{CHIfg}).

Next, by Theorem~\ref{CIIfg} every complete bipartite graph $K_{s,s}$ is $\CII$. So let $\phi$ be a monomorphism  between connected subgraphs of $K_{s,s}$. 
If $u \ne v$ are in the same part of the partition, then $\phi(u), \phi(v)$ are in the same part by Lemma~\ref{biparthom}, and $\phi(u) \ne \phi(v)$ since $\phi$ is injective. Thus all nonedges are preserved, and so $\phi$ is in fact an isomorphism. Hence $K_{s,s}$ is $\CMI$.

Similarly, by Theorem~\ref{CIIfg} every cycle $C_n$ is $\CII$. Now note that any connected proper subgraph of $C_n$ is a path, so any monomorphism between connected subgraphs of the same size must in fact be an isomorphism. Hence $C_n$ is $\CMI$.

\sk
We now show that the list is complete. So suppose that $G$ is a finite connected $\CMI$ graph. 
Then clearly $G$ is $\CII$, but rather than just checking through the $\CII$ graphs on Gardiner's list in Theorem~\ref{CIIfg} and finding the $\CMI$ graphs by inspection, we may show that the list is complete directly.

The main fact used is the following: if $G$ is $\CMI$ and has an $n$-cycle as an induced subgraph, then it has no induced subgraphs which are $(n - 1)$-paths. Otherwise suppose $G$ has an induced $(n - 1)$-path $a_1 a_2 \ldots a_n$ and an induced $n$-cycle $b_1 b_2 \ldots b_n b_1$. 
The monomorphism which maps $a_i$ to $b_i$ for $i \in [n]$ is not an isomorphism since the nonedge $a_1 \nsim a_n$ is mapped to the edge $b_1 \sim b_n$; so $G$ is not $\CMI$.

Thus if $G$ embeds an $n$-cycle, then it does not embed an $(n-1)$-path, and in particular it does not embed an $i$-cycle where $i > n$. 
Furthermore, $G$ embeds $(n-1-j)$-paths for $j \ge 1$, so it does not embed $(n-j)$-cycles for $j \ge 1$.
Thus in fact the only cycles that can be embedded in $G$ have length $n$. 
Now there are four different cases that arise, which we consider in turn. 

\sk
\noindent \textbf{Case 1:} \textit{$G$ does not embed any cycles.}
Then $G$ is a tree, and we claim that in fact $G$ must be $K_1$ or $K_2$. 
Otherwise, we can find $u \sim v$ in $G$ such that $u$ has degree 1 and $v$ has degree 2 or more; but clearly the monomorphism $u \mapsto v$ cannot be extended to an automorphism of $G$, contrary to $\CMI$-homogeneity of $G$. 

\sk
\noindent \textbf{Case 2:} \textit{$g(G) = 3$.}
$G$ embeds a triangle, and so $G$ does not embed any 2-paths, and thus $G$ is complete.

\sk
\noindent \textbf{Case 3:} \textit{$g(G) = 4$.}
$G$ embeds a square, and so as observed above, all induced cycles in $G$ are squares. So in particular $G$ will have no odd cycles, and hence it is bipartite. We claim that $G$ must be a complete bipartite graph $K_{s,s}$. 

If $G$ is not complete bipartite, then we can find $u \nsim v$ in different parts. Since $G$ is connected, there exists a shortest path from $u$ to $v$, which has length at least 3. But $G$ does not embed 3-paths, so we have a contradiction. 
Thus $G$ must be complete bipartite. 

Now suppose $G = K_{m,n}$ with $m \ne n$. Then clearly any isomorphism which maps a vertex from one part into the other cannot be extended to an automorphism. So $\CMI$ complete bipartite graphs must have parts of the same size.

\sk
\noindent \textbf{Case 4:} \textit{$g(G) = n \ge 5$.}
We claim that $G$ must in fact be isomorphic to $C_n$. 
Otherwise, without loss of generality, suppose $v_1 v_2 \ldots v_n v_1$ is an $n$-cycle in $G$, and $v \sim v_1$ with $v \notin \{v_1, \ldots, v_n \}$. 
Then $v \nsim v_i$ for $i \in \{2, \ldots, n\}$, otherwise (since $n \ge 5$) we could find a cycle of length less than $n$ that embeds in $G$. 
But now $v v_1 v_2 \ldots v_{n-1}$ is an induced $(n-1)$-path, which is a contradiction.

\sk
Hence if $G$ is a finite connected $\CMI$ graph, then it is either complete, complete bipartite with parts of the same size, or a cycle.
Thus we have characterised all finite $\CMI$ graphs.
\end{proof}

\section{Connected $\CHH$ graphs} \label{secCHHcfg}

In this section we prove Theorem~\ref{CHHcfg}, classifying the connected finite $\CHH$ graphs. 
We work through the two parts of the proof in turn --- beginning by showing that each of the graphs in the statement is indeed $\CHH$, and then moving to the harder task of showing that these are the only such graphs. 

\subsection{The graphs in the list are $\CHH$}

To show that a connected graph $G$ is $\CHH$, we may use \emph{one-point extensions} to inductively construct extension maps. 
If $\phi: A \to B$ is a homomorphism between connected subgraphs, then we show that for any $v \in G \setminus A$ such that $\langle A \cup \{v\} \rangle$ is connected, we can extend $\phi$ to a homomorphism from $A \cup \{v\}$ into $G$; and then since $G$ is countable and connected, we can inductively construct a homomorphism $\psi: G \to G$. So it is sufficient to show that if $\langle A \cup \{v\} \rangle$ is connected and $A_v := N(v) \cap A = \{ u \in A: v \sim u \}$, then there exists $v' \in G$ with $v' \sim \phi(A_v)$; so that defining $\phi(v) = v'$ gives the required homomorphism extension. 

Trivially $K_1$ is $\CHH$, so let us consider the other cases.

\begin{lem} \label{treelikeCHH} 			
If $G$ is a treelike graph, then it is $\CHH$.
\end{lem}

\begin{proof}
Let $G$ be a $K_n$-treelike graph ($n \ge 2$), let $\phi: A \to B$ be a homomorphism between connected subgraphs of $G$, and consider $v \in G \setminus A$ such that $\langle A \cup \{v\} \rangle$ is connected. 
Observe that since $A$ is connected, and there are no induced cycles in $G$ bigger than triangles, $N(v) \cap A$ will be completely contained in just one connected component of $N(v)$. 
So $\langle A_v \rangle := \langle N(v) \cap A \rangle$ will be a clique, with $1 \le |A_v| < n$; and since $\phi$ is a homomorphism, $\langle \phi(A_v) \rangle$ is a clique of the same size. Then we can always find some $v' \sim \phi(A_v)$, since $\phi(A_v) < n$ and all components of $G$ have size $n$. Note that if more than one such vertex exists, then we may just pick one. 
By extending $\phi$ to map $v$ to $v'$ we get a homomorphism as required. 
\end{proof}

\begin{lem} \label{allsqnotsCHH}			
If $G$ is a graph such that the only induced cycles in $G$ are squares but $\ts$ does not embed, then $G$ is $\CHH$. 
\end{lem}

\begin{proof}
First note that since $G$ has no odd induced cycles, $G$ is bipartite.

Let $\phi: A \to B$ be a homomorphism between connected subgraphs, and consider $v \in G \setminus A$ such that $\langle A \cup \{v\} \rangle$ is connected. 
Let $A_v := N(v) \cap A$, which is an independent set since $G$ is bipartite. 
We may assume that $|A_v| \ge 2$: otherwise $A_v = \{a\}$ but then clearly there exists $v' \in G$ with $v' \sim \phi(a) = \phi(A_v)$ as required, since $G$ is connected and nontrivial. 
We aim to show that there exists $a' \in A$ such that $a' \sim A_v$. 
Then by extending $\phi$ to map $v$ to $\phi(a')$ we get a homomorphism as required. 

First we show that every pair of vertices of $A_v$ has a common neighbour in $A$. 
So consider $a_1, a_2 \in A_v$; then since $A$ is connected, there exists a path in $A$ from $a_1$ to $a_2$, so let $P$ be a shortest such path. 
Suppose $P \cap A_v \ne \{ a_1, a_2 \}$, and consider the subpaths of $P$ between such vertices --- where each \emph{subpath} contains only two vertices of $A_v$, its end vertices. 
Consider a subpath $P'$ of $P$ between distinct $a_1', a_2' \in A_v$. 
Since $P$ was shortest possible, this is an induced path from $a_1'$ to $a_2'$. 
But then $\langle P' \cup \{ v \} \rangle$ is an induced cycle, and since the only induced cycles in $G$ are squares, each subpath has length 2. 
Now consider the first two subpaths $a_1 x a_1'$ and $a_1' x' a_2'$ of the path $P$. 
Then $\langle v, a_1, x, a_1', x', a_2' \rangle \cong \ts$, which contradicts the fact that $\ts$ does not embed in $G$. 
So in fact $P \cap A_v = \{ a_1, a_2 \}$, and the path $P$ has length 2; say $P = a_1 x a_2$, and then $x \in A$ is a common neighbour of $\{ a_1, a_2 \}$ as required. 

We now show that if $2 \le k < |A_v|$ and every $k$-subset of $A_v$ has a common neighbour in $A$, then so does every $(k+1)$-subset. 
Suppose for a contradiction that some $(k+1)$-subset $S$ of $A_v$ has no common neighbour. 
Now each of the $k+1$ $k$-subsets of $S$ has a common neighbour in $A$, and these are distinct (since $S$ does not have a common neighbour). 
Thus we have an induced subgraph of $G$ which is a copy of the bipartite complement of a perfect matching $\overline{L(K_{2,k+1})}$. 
But for $n \ge 3$, $C_6$ embeds in $\overline{L(K_{2,n})}$, which contradicts the fact that $C_6$ does not embed in $G$. 

Thus by induction, $A_v$ has a common neighbour $a'$ in $A$, as required. 
\end{proof}

Now observe that bipartite complements of perfect matchings, and bipartite graphs such that each part has a common neighbour, are characterised by the property that for each $k \le \Delta(G)$ every $k$-subset of a part has a common neighbour. It is then easy to see that bipartite graphs that satisfy this condition will be $\CHH$. 

\begin{lem} \label{B2graphs}			
Let $G$ be a bipartite graph. 
Then $G$ satisfies the property that for each $k \le \Delta(G)$ every $k$-subset of a part has a common neighbour if and only if $G$ is either a bipartite complement of a perfect matching, or each part of $G$ has a common neighbour. 
\end{lem}

\begin{proof}
The backward direction is an easy observation. 

For the forward direction, suppose $G$ has parts $X,Y$ with $|X| \ge |Y|$, and $d := \Delta(G) \ge 2$. 

Firstly, if $|X| > d$, then consider the set of $d$-subsets of $X$, which is denoted by $X^{\{d\}}$. Then $|X^{\{d\}}| = {|X| \choose d} \ge |X|$. 
Each of these $d$-subsets has a common neighbour in $Y$, and since $\Delta(G) = d$, each $y \in Y$ is adjacent to at most $d$ vertices of $X$. 
Let $Y' := \{ y \in Y : d(y) = d \} \subseteq Y$. 
Consider the map $f: Y' \to X^{\{d\}}$ given by $f(y) = N(y)$. 
Then $f$ is well-defined and surjective, so $|Y| \ge |Y'| \ge |X^{\{d\}}| \ge |X|$. 
But $|X| \ge |Y|$, so in fact $|Y| = |Y'| = |X^{\{d\}}| = |X|$. 
So $d = |X| - 1$ (since $d \ge 2$), and hence $G$ is the bipartite complement of a perfect matching. 

Otherwise $|X| \le d$, but since $|Y| \le |X|$ and $\Delta(G) = d$, we must have that $|X| = d$. 
So certainly there exists $y \in Y$ with $y \sim X$, and since $|Y| \le d$ there exists $x \in X$ with $x \sim Y$. That is, each part has a common neighbour. 
\end{proof}

\begin{lem} \label{B2CHH}			
Let $G$ be a bipartite graph. If for each $k \le \Delta(G)$ every $k$-subset of a part has a common neighbour, then $G$ is $\CHH$. 
\end{lem}

\begin{proof}
Let $\phi: A \to B$ be a homomorphism between connected subgraphs of $G$, and consider $v \in G$ with $\langle A \cup \{v\} \rangle$ connected. 
$N(v)$ is a subset of a part, of size at most $\Delta(G)$, and since $\phi$ preserves the partition (by Lemma~\ref{biparthom}), $A_v := N(v) \cap A$ and $\phi(A_v)$ are also subsets of a part of size at most $\Delta(G)$. Since for each $k \le \Delta(G)$ every $k$-subset of a part has a common neighbour, it follows that certainly $\phi(A_v)$ does; that is, there exists $v' \sim \phi(A_v)$. 
\end{proof}

\subsection{The list is complete}

The rest of this section is devoted to proving that the graphs in the statement of Theorem~\ref{CHHcfg} are in fact the only examples. 
So from now on, unless otherwise stated, $G$ will denote a nontrivial (that is, with at least two vertices) finite connected $\CHH$ graph. 

By Lemma~\ref{CXYsubneighbours}, for each $v \in G$, $N(v)$ is $\HH$, so by Theorem~\ref{MH=HHfg} it is a disjoint union of complete graphs of the same size. 
Furthermore, we require that for any $u, v \in G$, $N(u)$ and $N(v)$ are $\HH$-symmetric, so the size of the complete graphs is fixed by $G$ (it is straightforward to see that $t_1 \cdot K_{s_1}$, the disjoint union of $t_1$ copies of $K_{s_1}$, is $\HH$-symmetric to $t_2 \cdot K_{s_2}$ if and only if $s_1 = s_2$.) 

The problem of determining all possibilities for $G$ now naturally splits into two cases, which we consider separately. First the case that for each $v \in G$, $N(v)$ is a disjoint union of nontrivial complete graphs of the same size; and second the case that it is an independent set.

\sk
\noindent\textbf{Case 1. $N(v)$ a disjoint union of nontrivial complete graphs of the same size} 

First observe that since a triangle embeds in $G$, every edge is contained in a triangle (by $\CHH$-homogeneity). 
If for each $v \in G$, $N(v)$ is complete, then $G$ must itself be a complete graph. 
So we may assume that there is at least one vertex whose neighbour set is disconnected. 
In this case, we prove that $G$ is a treelike graph. 

\begin{lem}
Let $G$ be a $\CHH$ graph such that for each $v \in G$, $N(v)$ is a disjoint union of nontrivial complete graphs of size $s \ge 2$. Then the only induced cycles in $G$ are triangles, and so $G$ is a $K_{s+1}$-treelike graph. 
\end{lem}

\begin{proof}
If not, suppose that $C_k = v_1 v_2 \ldots v_k v_1$ is a minimal induced cycle in $G$ which is not a triangle (so $k \ge 4$), and aim for a contradiction. 
Since every edge is in a triangle, there exists $x \in G$ with $x \sim \{ v_1, v_2 \}$. 
We show that $x \nsim \{ v_3, \ldots, v_k \}$. 
Firstly, observe that $x \nsim v_3$, otherwise $N(x)$ embeds the 2-path $v_1 v_2 v_3$. Similarly $x \nsim v_k$.  
If $k > 4$, then suppose there is $i$ such that $3 < i < k$ and $x \sim v_i$ --- take $i$ to be the least such index. 
But then $x v_2 v_3 \ldots v_i x$ is an induced cycle (but certainly not a triangle) of length less than $k$, which contradicts the initial choice of $C_k$.

Thus $x v_2 \ldots v_k$ is an induced $(k-1)$-path. 
If $k$ is odd, then consider the homomorphism $\phi$ between connected subgraphs of $G$ which fixes $x$ and maps $v_i$ with $i$ even to $v_2$, and $v_i$ with $i >1$ odd to $v_3$. 
If $k$ is even, then consider the homomorphism $\phi$ which 
maps $x$ to $v_2$, $v_2$ to $x$, $v_i$ with $i > 2$ even to $v_3$, and $v_i$ with $i > 1$ odd to $v_2$. 
In either case, $v_1 \sim \{ x, v_2, v_k \}$ but $\phi( \{ x, v_2, v_k \} ) =  \langle x, v_2, v_3 \rangle$ which is a 2-path and so does not have a common neighbour. 
But then $\phi$ cannot be extended to a homomorphism. 

The final assertion follows from the definition of treelike graphs. 
\end{proof}

\noindent\textbf{Case 2. $N(v)$ an independent set} 

First we show that in this case, $G$ is bipartite. 

\begin{lem} 	
If $G$ is a connected $\CHH$ graph such that for each $v \in G$, $N(v)$ is an independent set, then $G$ is bipartite.
\end{lem}

\begin{proof}
Suppose for a contradiction that $G$ embeds an odd cycle $v_1 \ldots v_k v_1$. 
Since for each $v$ in $G$, $N(v)$ is an independent set, $K_3$ does not embed in $G$; so $k > 3$. 
Consider the homomorphism $\phi$ between connected subgraphs of $G$ which maps each $v_i$ with $i < k$ odd to $v_1$, and each $v_i$ with $i$ even to $v_2$. 
Now $v_k \sim \{ v_1, v_{k-1} \}$, but $\phi( \{ v_1, v_{k-1} \} ) = \{ v_1, v_2 \}$, which does not have a common neighbour since $K_3$ does not embed. But then $\phi$ cannot be extended to a homomorphism. 
\end{proof}

So in this case we prove that bipartite $G$ is one of the following: a tree; a graph such that all induced cycles are squares, but $\ts$ does not embed; the bipartite complement of a perfect matching; or a bipartite graph such that each part has a common neighbour. First we find a bound for the girth of $G$. 

\begin{lem} \label{girthCHH}			
If $G$ is a connected $\CHH$ graph which is not a tree, then $g(G) \le 6$. 
\end{lem}

\begin{proof}
Suppose that $g(G) \ge k > 6$, and aim for a contradiction. 
Let $v_1 \ldots v_k v_1$ be an induced $k$-cycle of $G$, and consider the homomorphism $\phi$ which fixes $v_i$ for $i \in [k-2]$, and maps $v_{k-1}$ to $v_{k-3}$. 
Let $\psi: G \to G$ be a homomorphism which extends $\phi$. 
Now $v_k \sim \{ v_1, v_{k-1} \}$, so $x := \psi(v_k) \sim \psi( \{v_1, v_{k-1} \}) = \{ v_1, v_{k-3} \}$. But then $\langle v_1, x, v_{k-3}, v_{k-2}, v_{k-1}, v_k \rangle$ is an induced 6-cycle or embeds a smaller cycle (if $x$ has other neighbours in this set), which contradicts the fact that $g(G) \ge 6$. 
\end{proof}

We begin with the case that neither $C_6$ nor $\ts$ embeds in $G$. 

\begin{lem} 		\label{CHHbipcycles}
If $G$ is a bipartite $\CHH$ graph such that $C_6$ and $\ts$ do not embed, then neither does any bigger even cycle. So $G$ is either a tree, or the only induced cycles of $G$ are squares (but $\ts$ does not embed).
\end{lem}

\begin{proof}
If not, suppose $v_1 v_2 \ldots v_{2k} v_1$ is an induced $2k$-cycle with $k > 3$. 
Consider $\phi$ which fixes $v_1, v_2, v_3, v_4, v_5$, and maps $v_{2i}$ to $v_4$ and $v_{2i+1}$ to $v_5$ for $2 \le i < k$. 
Now $v_{2k} \sim \{ v_1, v_{2k-1} \}$ and $\phi(\{ v_1, v_{2k-1} \}) = \{ v_1, v_5 \}$, but these do not have a common neighbour (as otherwise if $v' \sim \{ v_1, v_5 \}$, then $\langle v', v_1, \ldots, v_5 \rangle$ is an induced 6-cycle or a copy of $\ts$). 

If $G$ has no cycles at all, then it is a tree. 
Otherwise $G$ has some induced cycles and these are all squares.
\end{proof}

We remark that as well as complete bipartite graphs and trees, this case includes many other more varied graphs with girth 4 (which can have any finite diameter). For instance, if $e$ is an edge of $K_{m,n}$ and $G = K_{m,n} \setminus e$, then all induced cycles of $G$ are squares but $\ts$ does not embed (we may also note that each part of $G$ has a common neighbour). 
We can also construct new sorts of treelike graphs which satisfy this property. For instance, instead of using complete graphs as the components of a treelike graph, now use squares (that is, copies of $K_{2,2}$) to form \emph{$K_{2,2}$-treelike} graphs (by joining components exactly as before). This can be generalised further to form \emph{$K_{m,n}$-treelike} graphs (note that now, unlike for $K_n$-treelike graphs where each component has fixed size $n$, here $m,n$ may vary for different components within a $K_{m,n}$-treelike graph); observe that these graphs satisfy the property that all induced cycles are squares but $\ts$ does not embed. See Figure~\ref{treelikegpic3} for some examples.

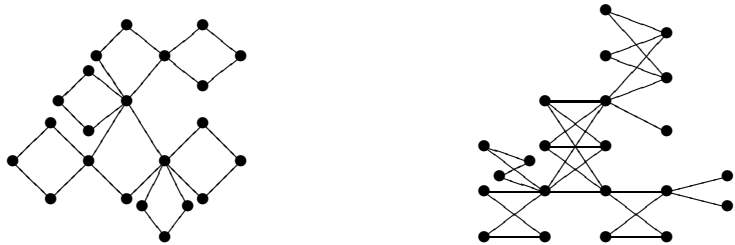
\begin{figure}[h!tb]		
\hspace{2.7cm}
\begin{xy}
(0,10)*={\bullet}="a" ,
(5,5)*={\bullet}="b" ,
(5,15)*={\bullet}="c" ,
(10,10)*={\bullet}="d" ,
(15,5)*={\bullet}="e" ,
(15,18)*={\bullet}="f" ,
(20,10)*={\bullet}="g" ,
(25,5)*={\bullet}="h" ,
(25,15)*={\bullet}="i" ,
(30,10)*={\bullet}="j" ,
(6,18)*={\bullet}="k" ,
(10,14)*={\bullet}="l" ,
(10,22)*={\bullet}="m" ,
(11,24)*={\bullet}="n" ,
(15,28)*={\bullet}="o" ,
(20,24)*={\bullet}="p" ,
(25,20)*={\bullet}="q" ,
(25,28)*={\bullet}="r" ,
(30,24)*={\bullet}="s" ,
(17,4)*={\bullet}="t" ,
(20,0)*={\bullet}="u" ,
(23,4)*={\bullet}="v" ,
"a";"b" **@{-} ,
"a";"c" **@{-} ,
"d";"b" **@{-} ,
"d";"c" **@{-} ,
"d";"e" **@{-} ,
"d";"f" **@{-} ,
"g";"e" **@{-} ,
"g";"f" **@{-} ,
"g";"h" **@{-} ,
"g";"i" **@{-} ,
"j";"h" **@{-} ,
"j";"i" **@{-} ,
"k";"l" **@{-} ,
"k";"m" **@{-} ,
"f";"l" **@{-} ,
"f";"m" **@{-} ,
"n";"o" **@{-} ,
"n";"f" **@{-} ,
"p";"o" **@{-} ,
"p";"f" **@{-} ,
"p";"q" **@{-} ,
"p";"r" **@{-} ,
"s";"q" **@{-} ,
"s";"r" **@{-} ,
"t";"u" **@{-} ,
"t";"g" **@{-} ,
"v";"u" **@{-} ,
"v";"g" **@{-} ,
(62,0)*={\bullet}="a1" ,
(62,6)*={\bullet}="a2" ,
(64,8)*={\bullet}="a3" ,
(62,12)*={\bullet}="a4" ,
(70,0)*={\bullet}="b1" ,
(70,6)*={\bullet}="b2" ,
(68,10)*={\bullet}="b23" ,
(70,12)*={\bullet}="b3" ,
(70,18)*={\bullet}="b4" ,
(78,0)*={\bullet}="c1" ,
(78,6)*={\bullet}="c2" ,
(78,12)*={\bullet}="c3" ,
(78,18)*={\bullet}="c4" ,
(78,24)*={\bullet}="c5" ,
(78,30)*={\bullet}="c6" ,
(86,0)*={\bullet}="d1" ,
(86,6)*={\bullet}="d2" ,
(86,14)*={\bullet}="d3" ,
(86,21)*={\bullet}="d4" ,
(86,27)*={\bullet}="d5" ,
(94,4)*={\bullet}="e1" ,
(94,8)*={\bullet}="e2" ,
"a1";"b1" **@{-} ,
"a1";"b2" **@{-} ,
"a2";"b1" **@{-} ,
"a2";"b2" **@{-} ,
"b2";"c2" **@{-} ,
"b2";"c3" **@{-} ,
"b2";"c4" **@{-} ,
"b3";"c2" **@{-} ,
"b3";"c3" **@{-} ,
"b3";"c4" **@{-} ,
"b4";"c2" **@{-} ,
"b4";"c3" **@{-} ,
"b4";"c4" **@{-} ,
"c1";"d1" **@{-} ,
"c1";"d2" **@{-} ,
"c2";"d1" **@{-} ,
"c2";"d2" **@{-} ,
"c4";"d4" **@{-} ,
"c4";"d5" **@{-} ,
"c5";"d4" **@{-} ,
"c5";"d5" **@{-} ,
"c6";"d4" **@{-} ,
"c6";"d5" **@{-} ,
"a3";"b2" **@{-} ,
"a3";"b23" **@{-} ,
"a4";"b2" **@{-} ,
"a4";"b23" **@{-} ,
"c4";"d3" **@{-} ,
"d2";"e1" **@{-} ,
"d2";"e2" **@{-} ,
\end{xy}
\caption{Examples of $K_{2,2}$-treelike and $K_{m,n}$-treelike graphs.} \label{treelikegpic3}
\end{figure}

We next observe that bipartite $\CHH$ graphs with diameter 3 (in particular $\CHH$ graphs that embed $C_6$ or $\ts$) satisfy the condition that all subsets of a part of size at most the maximum degree have a common neighbour. 

\begin{lem} \label{CHHdiam3}
Let $G$ be a bipartite $\CHH$ graph. 
\begin{enumerate}
\item[(i)] If at least one of $C_6$ or $\ts$ embeds in $G$, then $\diam(G)=3$.  
\item[(ii)] If $\diam(G)=3$, then for each $k \le \Delta(G)$ every $k$-subset of a part has a common neighbour. 
\end{enumerate}
\end{lem}

\begin{proof}
For the first part, observe that if either $C_6$ or $\ts$ embeds in $G$, then $G$ has an induced $4$-path such that the end vertices of the path have a common neighbour. Then by $\CHH$-homogeneity, clearly $\diam(G) = 3$.

For the second part, suppose that $\diam(G) = 3$. First note that if $\Delta(G) = 2$, then $G$ is either a 6-cycle or a 3-path, which certainly satisfy this condition; so we may assume that $\Delta(G) \ge 3$. 
Now suppose for a contradiction that $k \le \Delta(G)$ is the least integer such that there is a $k$-subset of a part which does not have a common neighbour. Note that since $\diam(G) = 3$, every pair of vertices in the same part has a common neighbour, so certainly $k > 2$. 

Let bipartite $G$ have parts $X,Y$, and without loss of generality suppose that $X' := \{ x_1, x_2, \ldots, x_k \}$ is a $k$-subset of $X$ without a common neighbour. 
Now each of the $k$ $(k-1)$-subsets of $X'$ has a common neighbour, and these are distinct (since $X'$ does not have a common neighbour). 
That is, there exists a $k$-subset $Y' := \{ y_1, y_2, \ldots, y_k \}$ of $Y$ such that $y_i \sim x_j$ if and only if $j \in [k] \setminus \{ i \}$; so $\langle X' \cup Y \rangle$ is a copy of the bipartite complement of a perfect matching $\overline{L(K_{2,k})}$. 
Since $\Delta(G) = k$, certainly this is a proper subgraph of $G$; so without loss of generality, say we have $x' \sim y_k$ with $x' \notin X'$. 
If $x'$ is adjacent to some nonempty subset of $\{ y_1, \ldots, y_{k-1} \}$, then the map $\phi$ which fixes $x_1, \ldots, x_k, y_1, \ldots, y_{k-1}$ and maps $x'$ to $x_k$ is a map between connected subgraphs of $G$, and since $x_k \sim \{ y_1, \ldots, y_{k-1} \}$ it is a homomorphism. 
Otherwise, if $x' \nsim \{ y_1, \ldots, y_{k-1} \}$, then certainly there exists some common neighbour $y'$ of $x'$ and $x_k$; so by also defining $\phi$ to fix $y'$, the subgraphs are again connected and this is a homomorphism as before. 
In either case $\phi$ maps the $k$-subset $\{ x_1, \ldots, x_{k-1}, x' \}$ with common neighbour $y_k$ to the $k$-subset $X'$ which does not have a common neighbour; so $\phi$ cannot be extended for $y_k$, and hence $G$ is not $\CHH$. 
\end{proof}

Finally, by Lemma~\ref{B2graphs} the only bipartite graphs that satisfy this condition are bipartite complements of perfect matchings, and those such that each part has a common neighbour. This completes the analysis of all possibilities for a finite connected $\CHH$ graph, and hence the proof of Theorem~\ref{CHHcfg}.

\section{The class $\CHH$} \label{secCHHfg}

To complete a classification of the finite $\CHH$ graphs, we must consider the disconnected ones.
So in view of Proposition~\ref{discXYrs}, we wish to determine when two connected $\CHH$ graphs $G_1,G_2$ are $\CHH$-symmetric.

Recall (Remark~\ref{symmetric}) that if $G_1,G_2$ are $\CHH$-symmetric, then they must be homomorphically equivalent, that is, have the same core. 
In particular, (non-)bipartite $\CHH$ graphs can only possibly be $\CHH$-symmetric to other (non-)bipartite $\CHH$ graphs, since all bipartite graphs are homomorphically equivalent, with core $K_2$. 
We begin with the $\CHH$ graphs that are not bipartite; in this case it is straightforward to characterise those that are $\CHH$-symmetric. 
First observe that the only connected graph with core $K_1$ is $K_1$ itself, so the only connected $\CHH$ graph that $K_1$ is $\CHH$-symmetric to is itself; so from now on we may assume that all connected components are nontrivial. 

\begin{lem} \label{treelikeCHHsym} 			
If $G_1,G_2$ are both $\CHH$ and $G_1$ is a $K_n$-treelike graph with $n \ge 3$, then they are $\CHH$-symmetric if and only if $G_2$ is also $K_n$-treelike.
\end{lem}

\begin{proof}
The proof that two $K_n$-treelike graphs are $\CHH$-symmetric is essentially the same as showing that a $K_n$-treelike graph is $\CHH$ (Lemma \ref{treelikeCHH}). 
Conversely, if $G_1,G_2$ are $\CHH$-symmetric, then they must have the same core. It is easily seen that the core of a $K_n$-treelike graph is $K_n$. So if $G_1$ is $K_n$-treelike with $n \ge 3$, then its core is $K_n$, so $G_2$ must also have core $K_n$. Hence $G_2$ must also be $K_n$-treelike, since by Theorem~\ref{CHHcfg} these are the only $\CHH$ graphs with core $K_n$.
\end{proof}

We now consider the bipartite case. So from now on, let $G_1, G_2$ be finite bipartite $\CHH$ graphs. 
Recall that by Theorem~\ref{CHHcfg}, and the previous section, there are a number of different families of bipartite $\CHH$ graphs. 
For ease of reference, we introduce the following notation for the characterising properties of these bipartite graphs: 
\begin{itemize}
\item[] (B1) all induced cycles are squares and $\ts$ does not embed;

\item[] (B2) for each $k \le \Delta(G)$ every $k$-subset of a part has a common neighbour;

\item[] (B2*) each part of $G$ has a common neighbour.
\end{itemize}
Note that by Lemmas~\ref{CHHbipcycles} and \ref{CHHdiam3}, if $G$ is a connected bipartite $\CHH$ graph then either (B1) or (B2) holds (and possibly both); and note that trees satisfy (B1) vacuously. Furthermore, by Lemma~\ref{B2graphs}, if $G$ satisfies (B2), then in fact either (B2*) holds or $G$ is a bipartite complement of a perfect matching (and then (B2*) does not hold).
We also note that $C_6 = \overline{L(K_{2,3})}$, so $C_6$ embeds in each connected bipartite complement of a perfect matching, so then these graphs do not satisfy (B1). Thus if $G$ satisfies (B1) and (B2), then in fact it must satisfy (B2*).

First observe that any two bipartite $\CHH$ graphs which satisfy (B1) are certainly $\CHH$-symmetric. 

\begin{lem} \label{B1CHHsym} 			
If for both bipartite $\CHH$ graphs $G_1,G_2$ all induced cycles are squares and $\ts$ does not embed, then they are $\CHH$-symmetric.
\end{lem}

\begin{proof}
Use exactly the same argument as in the proof of Lemma~\ref{allsqnotsCHH}, except here consider initial homomorphisms between connected subgraphs $A,B$ of $G_1, G_2$.
\end{proof}

Otherwise, at least one of the graphs does not satisfy (B1). Without loss of generality suppose that either $C_6$ or $\ts$ embeds in $G_1$. By Lemma~\ref{CHHdiam3}, $G_1$ must satisfy property (B2), that is for each $k \le \Delta(G_1)$ every $k$-subset of a part has a common neighbour. 
In this case, for $G_1, G_2$ to be $\CHH$-symmetric we find that $G_2$ must also satisfy (B2).

\begin{lem} \label{B1B2CHHsym} 			
If the bipartite $\CHH$ graphs $G_1, G_2$ are $\CHH$-symmetric and $G_1$ embeds $C_6$ or $\ts$, then $G_2$ also satisfies (B2). 
\end{lem}

\begin{proof}
We claim that $\diam(G_2) \le 3$, so then by Lemma \ref{CHHdiam3}, $G_2$ satisfies (B2). 
Otherwise, suppose that $\diam(G_2) > 3$, and let $a,b \in G_2$ be at distance 4. Since $G_1$ embeds $C_6$ or $\ts$, certainly we can find an induced 4-path whose end vertices $x,y$ have a common neighbour $z$. Consider $\phi$ which maps the 4-path from $x$ to $y$ onto the 4-path from $a$ to $b$; this is a homomorphism from a connected subgraph of $G_1$ into $G_2$. Now any extension of $\phi$ (which exists since $G_1, G_2$ are $\CHH$-symmetric) maps $z$ to a common neighbour of $a$ and $b$ --- contradicting the fact that these are at distance 4. 
\end{proof}

Thus if bipartite $G_1, G_2$ are $\CHH$-symmetric, but do not both satisfy (B1), then in fact they both satisfy (B2). 
So we are just left with determining which graphs $G_1,G_2$ that both satisfy (B2) are $\CHH$-symmetric. 
Let us divide this case into a number of smaller cases which cover all possibilities:
\begin{enumerate}
\item $G_1, G_2$ do not satisfy (B2*) (i.e.\ they are both bipartite complements of perfect matchings);
\item $\Delta(G_1) = \Delta(G_2)$;
\item $\Delta(G_1) > \Delta(G_2)$, and $G_2$ satisfies (B2*);
\item $\Delta(G_1) > \Delta(G_2)$, $G_1$ satisfies (B2*) and $G_2$ does not.
\end{enumerate}

We begin with the first case. 

\begin{lem} \label{bcpmCHHsym} 			
If $G_1, G_2$ are both bipartite complements of perfect matchings, then they are $\CHH$-symmetric if and only if they are in fact isomorphic.
\end{lem}

\begin{proof}
Let $G_1 = \overline{L(K_{2,n_1})}$, $G_2 = \overline{L(K_{2,n_2})}$ with $n_1 > n_2$. 
So $G_2$ is isomorphic to the induced subgraph $\langle x_1, \ldots, x_{n_2}, y_1, \ldots, y_{n_2} \rangle \subset G_1$. 
But there is no homomorphism $G_1 \to G_2$ which fixes this subgraph (for example $x_{n_1} \sim \{ y_1, \ldots, y_{n_2} \}$, but none of $x_1, \ldots, x_{n_2}$ satisfy this).
Hence $G_1$ is not $\CHH$-morphic to $G_2$.

The converse is obvious since any bipartite complement of a perfect matching is $\CHH$.
\end{proof}

Now consider the next two cases. Here we find that there are no further restrictions. In either case the graphs will be $\CHH$-symmetric.  

\begin{lem} \label{B2degCHHsym} 			
Suppose $G_1, G_2$ are $\CHH$ graphs which both satisfy (B2). 
\begin{enumerate}
\item[(i)] If $\Delta(G_1) \le \Delta(G_2)$, then $G_1$ is $\CHH$-morphic to $G_2$.
\item[(ii)] If $\Delta(G_1) = \Delta(G_2)$, then they are $\CHH$-symmetric.
\item[(iii)] If $\Delta(G_1) > \Delta(G_2)$, and $G_2$ satisfies (B2*), then they are $\CHH$-symmetric.
\end{enumerate}
\end{lem}

\begin{proof}
In each case we look at extending partial maps from $G_1$ to $G_2$, generalising the proof of Lemma~\ref{B2CHH}. So let $\phi$ be a homomorphism which maps the connected subgraph $A \subset G_1$ into $G_2$, and consider $v \in G_1$ with $\langle A \cup \{v\} \rangle$ connected. 
$N(v)$ is a subset of a part of $G_1$, and so $A_v := N(v) \cap A$ is too, and since $\phi$ preserves the partitions (by Lemma~\ref{biparthom}), $\phi(A_v)$ is a subset of a part of $G_2$. 

\begin{enumerate}
\item[(i)]
In this case $|\phi(A_v)| \le |A_v| \le |N(v)| \le \Delta(G_1) \le \Delta(G_2)$, because $\phi$ is a homomorphism and $A_v \subseteq N(v)$.
Now since $G_2$ satisfies (B2), for each $k \le \Delta(G_2)$ every $k$-subset of a part of $G_2$ has a common neighbour. So certainly $\phi(A_v)$ does; that is, there exists $v' \in G_2$ with $v' \sim \phi(A_v)$ as required. 

\item[(ii)] 
Follows directly from (i). 

\item[(iii)]
Clearly in this case, since $G_2$ satisfies (B2*), there exists a common neighbour of each part; that is, there exists $v' \in G_2$ with $v' \sim \phi(A_v)$ as required.
Thus $G_1$ is $\CHH$-morphic to $G_2$. 
Also $G_2$ is $\CHH$-morphic to $G_1$ by (i), so they are in fact $\CHH$-symmetric.
\end{enumerate}
\end{proof}

The final case is more complicated. 
In this case, $G_1, G_2$ are bipartite $\CHH$ graphs which both satisfy (B2), with $\Delta(G_1) > \Delta(G_2)$, such that $G_1$ satisfies (B2*) and $G_2$ does not (so $G_2$ is a bipartite complement of a perfect matching). 
So say $G_2 =  \overline{L(K_{2,n})}$, then $\Delta(G_2) = \Delta(\overline{L(K_{2,n})}) = n-1$; and 
let $G_1$ be a bipartite graph such that each part has a common neighbour and $\Delta(G_1) \ge n$, that is, at least one of its parts has size at least $n$.  
Note that by Lemma~\ref{B2degCHHsym}(i), we know that $\overline{L(K_{2,n})}$ is $\CHH$-morphic to $G_1$, so we just need to determine when $G_1$ is $\CHH$-morphic to $\overline{L(K_{2,n})}$, and when it is not. 

Recall the definition of a $\PCM(n)$ graph from Section~\ref{secintro}.

\begin{lem} \label{B2mixCHHsym} 			
Suppose that $G_1$ satisfies (B2*) with $\Delta(G_1) \ge n$, and $G_2 =  \overline{L(K_{2,n})}$. 
Then $G_1$ is $\CHH$-symmetric to $G_2$ if and only if $G_1$ is $\PCM(n)$-free. 
\end{lem}

\begin{proof}
Let the parts of $G_2 = \overline{L(K_{2,n})}$ be $X,Y$ and label the vertices as previously. 

By Lemma~\ref{B2degCHHsym}(i), certainly $\overline{L(K_{2,n})}$ is $\CHH$-morphic to $G_1$, so $G_1$ is $\CHH$-symmetric to $\overline{L(K_{2,n})}$ if and only if $G_1$ is $\CHH$-morphic to $\overline{L(K_{2,n})}$. 

First we show that if $G_1$ embeds a $\PCM(n)$ graph, then $G_1$ is not $\CHH$-morphic to $\overline{L(K_{2,n})}$. 
Let $H$ be a $\PCM(n)$ graph which embeds in $G_1$, with parts $Z = \{ z_1, \ldots, z_m \}, W = \{ w_1, \ldots, w_n\}$ such that $2 \le m \le n$, and $z_i \nsim w_i$ for $i \in [m]$. 
Now let $\phi : z_i \mapsto x_i$ for $i \in [m]$, $w_i \mapsto y_i$ for $i \in [n]$. 
Then $\phi$ is a homomorphism from a connected subgraph of $G_1$ into $\overline{L(K_{2,n})}$. 
But there is some $z \in G_1$ with $z \sim W$ since each part of $G_1$ has a common neighbour, and $\phi$ cannot be extended for $z$ since $Y$ does not have a common neighbour in $\overline{L(K_{2,n})}$. 

\sk
Next we show that if $G_1$ is not $\CHH$-morphic to $\overline{L(K_{2,n})}$, then $G_1$ embeds a $\PCM(n)$ graph. 
So suppose $G_1$ is not $\CHH$-morphic to $\overline{L(K_{2,n})}$. Then there is a map that does not extend. 
That is, there is a connected $A \subset G_1$ with a homomorphism $\phi: A \to G_2$, and $v \in G_1$ for which $A_v := N(v) \cap A$ is nonempty, such that $\phi(A_v)$ has no common neighbour in $G_2$. 
Now recall that $A_v$ must be a subset of a part, and since $\phi$ preserves the partition, so is $\phi(A_v)$. Now the only subset of a part of $G_2 = \overline{L(K_{2,n})}$ which does not have a common neighbour is a whole part. So for such a map, $|\phi(A_v)| = n$, and without loss of generality say that $\phi(A) \cap Y = Y = \phi(A_v)$. 
Let $A$ have parts $Z,W$; and without loss of generality say that $\phi(W) = Y$. 

If $A$ is itself a $\PCM(n)$ graph, then we are done. Otherwise, we now show how to find a subgraph $A' \subset A$ which is a $\PCM(n)$ graph, by means of the following algorithm:

\begin{description}
\item[Initial step $1$:] 
Begin with $z_1 \in Z$ of maximal degree.  
Now find a 3-path $z_1 w_0 z_2 w_1$ such that $\{ w \in N(z_2) : w \nsim z_1 \}$ is as large as possible. 
Choose $w_2 \in N(z_1)$ with $z_2 \nsim w_2$. 
Put $Z_1 = \{ z_1, z_2 \}$, $W_1 = \{ w_0, w_1, w_2 \}$. 
Go to step 2.

\item[Iterative step $i$:] 
Suppose we have constructed a connected bipartite graph $A_i$ with parts $Z_i, W_i$. 
We consider $N^*(Z_i) :=  \underset{z \in Z_i}\bigcup N(z) \subseteq W$. 

\begin{itemize}
\item
If $|N^*(Z_i)| < n$, then find a 2-path $w' z_{i+2} v_i$ with $w' \in N^*(Z_i)$, $v_i \in W \setminus N^*(Z_i)$, such that $\{ w \in N(z_{i+2}) : w \nsim Z_i \}$ is as large as possible. 
Put $Z_{i+1} := Z_i \cup \{ z_{i+2} \}$, $W_{i+1} = N^*(Z_i) \cup \{ v_i \}$. 
Go to step $i+1$.

\item
If $|N^*(Z_i)| \ge n$, then stop. 
Let $S_i \subseteq N^*(Z_i) \setminus W_i$ be such that $|S_i| = n - |W_i|$. 
Put $W' := W_i \cup S_i$, $Z' := Z_i$, and let $A'$ be the induced subgraph of $A$ with parts $Z', W'$; and output $A'$. 
\end{itemize}
\end{description}

We now verify that the algorithm does indeed produce a $\PCM(n)$ graph. We use induction to show that the bipartite graph $A_{i+1}$ produced in the $i$th iteration of the algorithm is a $\PCM(|W_{i+1}|)$ graph where $|W_{i+1}| > |W_{i}|$. 
We also verify that we can indeed find the vertices as claimed in the initial and iterative steps. 

\sk
\noindent \textbf{Base case:}
For the base case, we begin by showing that we can construct $A_1$ as claimed. 

First note that $A$ has the property that for each $z \in Z$ there exists $w \in W$ such that $z \nsim w$. 
Consider $z \in Z$, and say $\phi(z) = x_i$. 
Since $\phi$ maps $W$ onto $Y$, there exists $w \in W$ such that $\phi(w) = y_i$. 
Then $z \nsim w$ since $x_i \nsim y_i$ and $\phi$ preserves edges (and preserves the partition). 

So we begin by choosing $z_1 \in Z$ with maximal degree, say $d(z_1) = m$. 
Then there is $w \in W$ such that $z_1 \nsim w$. 
In fact since $A$ is connected, there exists $w \in W$ with $\dist (w, z_1) = 3$. 
So we consider all 3-paths $z_1 w' z w$ in $A$, and choose one such that $\{ w \in N(z) : w \nsim z_1 \}$ is as large as possible; and let this path be $z_1 w_0 z_2 w_1$. 
Now observe that there exists $w_2 \in N(z_1)$ with $z_2 \nsim w_2$ as claimed, since $z_2 \sim w_1$ but $z_1 \nsim w_1$, and $z_1$ had maximal degree. 
So we have constructed $A_1$. 

By the construction $A_1 = \langle z_1, z_2, w_0, w_1, w_2 \rangle$. 
This is a bipartite graph with parts of size $|Z_1| = 2 < |W_1| = 3$. 
Observe that $w_0 \sim \{ z_1, z_2 \}$, $w_1 \sim z_2$, $w_2 \sim z_1$, so it is connected. 
Furthermore $A_1$ has a perfect complement matching: $z_1 \nsim w_1$, $z_2 \nsim w_2$. 
Thus $A_1$ is a $\PCM(3)$ graph. 

\sk
\noindent \textbf{Induction hypothesis:} Suppose $i \ge 1$ and that $A_{i}$ is a $\PCM(|W_{i}|)$ graph. 

\sk
\noindent \textbf{Induction step:}
First we verify that we can find the vertices $z_{i+2}, v_i$ for the iterative step as claimed. 
If $|N^*(Z_i)| < n$, then certainly $W \setminus N^*(Z_i) \ne \emptyset$ (we know $|W| \ge n$ since $\phi$ maps $W$ onto $Y$). 
Then $Z_i \ne Z$ since $A$ is connected. 
So there exists $w \in W \setminus N^*(Z_i)$; and since $A$ is connected, in particular there exists $w \in W \setminus N^*(Z_i)$ with $\dist(\overline{w}, w) = 2$ for some $\overline{w} \in N^*(Z_i)$. 
Then we may choose such a 2-path $w' z_{i+2} v_i$ with $w' \in N^*(Z_i)$, $v_i \in W \setminus N^*(Z_i)$, $z_{i+2} \in Z \setminus Z_i$ such that $\{ w \in N(z_{i+2}) : w \nsim Z_i \}$ is as large as possible. 

Now observe that in each iteration, vertices are added to $A_i$ but none are taken away, so the graph grows as claimed. 
$|Z_i| = i+1$ for each $i$, since $|Z_1| = 2$ and at each iteration exactly one vertex is added to this part. 
Meanwhile, $|W_i| \ge i+2$ for each $i$, since $|W_1| = 3$ and at each iteration at least one vertex is added to this part, namely $v_i$. So $|Z_i| < |W_i| < |W_{i+1}|$ for each $i$. 

In fact, $A_{i+1} \setminus A_i = \{ z_{i+2}, v_i \} \cup T_i$ where $T_i = N^*(Z_i) \setminus W_i$. 
In particular $T_i = \{ w \in N(z_{i+1}) : w \nsim Z_{i-1} \} \setminus \{ v_{i-1} \}$ for $i \ge 2$, 
and $T_1 = N(z_1) \cup N(z_2) \setminus \{ w_0, w_1, w_2 \}$. 
Now we may observe that since $A_{i}$ is connected, so is $A_{i+1}$. 
Clearly for each $w \in N^*(Z_i)$ there exists $z \in Z_i$ with $w \sim z$, so all such new vertices added in this iteration are connected to $Z_i \subset A_i$. 
Meanwhile, $z_{i+2} \sim w' \in N^*(Z_i) \subset W_{i+1}$ and $v_i \sim z_{i+2}$. Thus $A_{i+1}$ is connected. 

Next we show that $A_{i+1}$ has a perfect complement matching. 
So suppose that $A_{i}$ has a perfect complement matching given by $z_j \nsim w_j$ for $j \in [i+1]$. 
Observe that $|N^*(Z_i)| > m$, and so since $d(z_{i+2}) \le m$, there exists $w \in N^*(Z_i) \subset W_{i+1}$ with $z_{i+2} \nsim w$. 
If there exists $w \in N^*(Z_i) \setminus \{ w_1, \ldots, w_{i+1} \}$ with $z_{i+2} \nsim w$, then let $w_{i+2} := w$. 
Otherwise, pick $l \in [i+1]$ such that $w_l \nsim z_{i+2}$; and let $w_{i+2} := w_l$, $w_l := v_i$ (that is, relabel and reassign $w_l$). 
We know that $v_i \nsim z_j$ for $j \in [i+1]$, since $v_i \notin N^*(Z_i)$; so indeed $z_l \nsim v_i$. 
Then $z_j \nsim w_j$ for each $j \in [i+2]$, so this is a perfect complement matching of $A_{i+1}$. 

Then $A_{i+1}$ is a $\PCM(|W_{i+1}|)$ graph: we have seen that it is connected; it has a perfect complement matching $z_j \nsim w_j$ for $j \in [i+2]$; and $2 < i+2 = |Z_{i+1}| < |W_{i+1}|$. 

\sk
Finally we show that $A'$ is a $\PCM(n)$ graph. 
Suppose that $Z' = Z_k$. Then $A' = A_k \cup S_k$ where $S_k \subset N^*(Z_k)$. 
Note that for each $w \in S_k$ there exists $z_j \in Z_k$ with $w \sim z_j$; thus since $A_k$ is connected, so is $A'$. 
Furthermore, $A_k$ has a perfect complement matching $z_j \nsim w_j$ for each $j \in [k+1]$; so this is also a perfect complement matching for $A'$. 
Finally, note that $|W'| = |W_k| + |S_k| = n$, $|Z'| = k+1 \ge 2$, and certainly $|Z'| < |W'|$ since $|Z'| = |Z_k| < |W_k| \le |W'|$. 
So $A'$ is indeed a $\PCM(n)$ graph. 
\end{proof}

The condition of being $\PCM(n)$-free is a somewhat unsatisfactory characterising property, since it is not particularly easy to see whether or not an arbitrary bipartite graph is $\PCM(n)$-free. 
For each $n \in \N$, the list of finite connected graphs that are forbidden to embed in a $\PCM(n)$-free graph is certainly finite, but the list quickly gets large. 
However in the next result we see that if $G_1$ is not $\PCM(n)$-free, then certainly $\ts$ embeds in $G_1$. So all graphs satifying (B2*) which do not embed $\ts$ are $\PCM(n)$-free, and hence $\CHH$-symmetric to $\overline{L(K_{2,n})}$. This is useful since it is not so hard to verify that a graph does not embed the single graph $\ts$.

\begin{cor}
Suppose that $G_1$ satisfies (B2*) with $\Delta(G_1) \ge n$, and $G_2 =  \overline{L(K_{2,n})}$. 
If $\ts$ does not embed in $G_1$, then $G_1, G_2$ are $\CHH$-symmetric. 
\end{cor}

\begin{proof}
It suffices to show that if $G_1$ embeds a $\PCM(n)$ graph, then $\ts$ embeds in $G_1$. 
Suppose for a contradiction that $G_1$ embeds a $\PCM(n)$ graph $H$, with parts $Z,W$ such that $2 \le |Z| \le |W| = n$, but $\ts$ does not embed in $G_1$. 

First we show that $H$ does not embed $2 \cdot K_2$. Otherwise suppose we can find $z', z'' \in Z$, $w', w'' \in W$ with $z' \nsim w', z'' \nsim w''$ and $z' \sim w'', z'' \sim w'$; that is, $\langle z', z'', w', w'' \rangle = 2 \cdot K_2$. There exist $z, w \in G_1$ with $z \sim W$ and $w \sim Z$ since $G_1$ satisfies (B2*). But then $\langle z, z', z'', w, w', w'' \rangle = \ts$; a contradiction. 

However, we will see that in fact, all $\PCM(n)$ graphs embed $2 \cdot K_2$; obtaining the contradiction.

Let us attempt to construct a $\PCM(n)$ graph, $H$, which does not embed $2 \cdot K_2$.
Start with a vertex $w_1 \in W$. 
Since $H$ is connected, there exists $z_1 \in Z$ with $z_1 \sim w_1$. 
Since $H$ has a perfect complement matching, there exists $w_2 \in W$ with $z_1 \nsim w_2$. 
Since $H$ is connected, there exists $z_2 \in Z$ with $z_2 \sim w_2$. 
Now $z_2 \sim w_1$, otherwise $\langle z_1, z_2, w_1, w_2 \rangle$ is $2 \cdot K_2$. 

Since $H$ has a perfect complement matching, there exists $w_3 \in W$ with $z_2 \nsim w_3$. 
Now $w_3 \nsim z_1$, otherwise $\langle z_1, z_2, w_2, w_3 \rangle$ is $2 \cdot K_2$.
Since $H$ is connected, there exists $z_3$ with $z_3 \sim w_3$. 
Now $z_3 \sim w_1$, otherwise $\langle z_2, z_3, w_1, w_3 \rangle$ is $2 \cdot K_2$; and $z_3 \sim w_2$, otherwise $\langle z_2, z_3, w_2, w_3 \rangle$ is $2 \cdot K_2$. 

Since $H$ has a perfect complement matching, there exists $w_4 \in W$ with $z_3 \nsim w_4$. 
Now $w_4 \nsim z_i$ for $i \in [2]$, otherwise $\langle z_i, z_3, w_3, w_4 \rangle$ is $2 \cdot K_2$. 
Since $H$ is connected, there exists $z_4 \in Z$ with $z_4 \sim w_4$. 
Now $z_4 \sim w_i$ for $i \in [3]$, otherwise $\langle z_3, z_4, w_i, w_4 \rangle$ is $2 \cdot K_2$. 

And so on.
Clearly this is not a finite construction (we construct an infinite bipartite graph with vertices $\{ z_i, w_i: i \in \N \}$ such that $z_i \sim w_j$ for $j \le i$, and $z_i \nsim w_j$ for $j > i$). Thus it is not possible to construct a $\PCM(n)$ graph (which is finite by definition) which does not embed $2 \cdot K_2$. 

So all $\PCM(n)$ graphs embed $2 \cdot K_2$, and so if $G_1$ (satisfying (B2*)) embeds a $\PCM(n)$ graph, then $\ts$ embeds in $G_1$.

Hence if $\ts$ does not embed in $G_1$, then certainly $G_1$ does not embed a $\PCM(n)$ graph, and so $G_1, G_2$ are $\CHH$-symmetric. 
\end{proof}

Following the remarks before this result, note that the class of (B2*) graphs which do not embed $\ts$ is only a subclass of the class which are $\PCM(n)$-free. So this gives us a nice class of graphs which are $\CHH$-symmetric to $\overline{L(K_{2,n})}$, but not all of them. For $n > 3$, there do exist graphs $G$ which satisfy (B2*) with $\Delta(G) \ge n$ which embed $\ts$, but which are $\PCM(n)$-free (and by Theorem~\ref{B2mixCHHsym} such graphs are $\CHH$-symmetric to $\overline{L(K_{2,n})}$). For example, the bipartite graph $G$ shown in Figure~\ref{PCMnfreepic} with parts $A,B$ such that $|A|=3, ~|B| = n > 3$, which is complete bipartite except $a_1 \nsim b_1, ~a_1 \nsim b_3, a_2 \nsim b_2, ~a_2 \nsim b_3$. Note $a_3 \sim B, ~b_n \sim A$ and $\langle a_1, a_2, a_3, b_1, b_2, b_4 \rangle = \ts$, so $G$ satisfies (B2*) and embeds $\ts$ as claimed. To see that $G$ is $\PCM(n)$-free, first observe that since $a_3$ is adjacent to the whole of $B$ it can not be part of any induced subgraph which has a perfect complement matching spanning the whole of $A$. But no connected induced subgraph of $G \setminus a_3$ spans the whole of $B$ since $b_3 \nsim \{a_1, a_2\}$, so no such connected induced subgraph is a $\PCM(n)$ graph (since $|B|=n$). 

\begin{figure}[h!tb]		
\hspace{6cm}
\begin{xy}
(10,30)*={\bullet}="a1" ,
(10,25)*={\bullet}="a2" ,
(10,20)*={\bullet}="a3" ,
(20,30)*={\bullet}="b1" ,
(20,25)*={\bullet}="b2" ,
(20,20)*={\bullet}="b3" ,
(20,15)*={\bullet}="b4" ,
(20,5)*={\bullet}="bn" ,
"b1";"a2" **@{-} ,
"b1";"a3" **@{-} ,
"b2";"a1" **@{-} ,
"b2";"a3" **@{-} ,
"b3";"a3" **@{-} ,
"b4";"a1" **@{-} ,
"b4";"a2" **@{-} ,
"b4";"a3" **@{-} ,
"bn";"a1" **@{-} ,
"bn";"a2" **@{-} ,
"bn";"a3" **@{-} ,
(7,30)*={a_1} ,
(7,25)*={a_2} ,
(7,20)*={a_3} ,
(23,30)*={b_1} ,
(23,25)*={b_2} ,
(23,20)*={b_3} ,
(23,15)*={b_4} ,
(20,11)*={\vdots} ,
(23,5)*={b_n} 
\end{xy}
\caption{Example of a (B2*) graph $G$ with $\Delta(G) = n$ which embeds $\ts$ and is $\PCM(n)$-free.} \label{PCMnfreepic}
\end{figure}
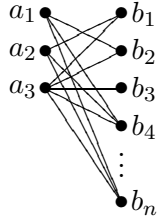

Despite these issues with actually recognising $\PCM(n)$-free graphs, it is still a perfectly good characterising property. We may now complete the classification of the finite $\CHH$ graphs by proving Theorem~\ref{CHHfg}.

\begin{proof}[Proof of Theorem~\ref{CHHfg}]
First observe that all graphs in the list are indeed $\CHH$. 
By Proposition~\ref{discXYrs}, if $G$ is a finite disconnected $\CHH$ graph, then all connected components of $G$ must be $\CHH$, and pairwise $\CHH$-symmetric. These properties hold for all the graphs in the list by Theorem~\ref{CHHcfg}, and Lemmas~\ref{treelikeCHHsym}, \ref{B1CHHsym}, \ref{B2degCHHsym}, \ref{bcpmCHHsym} and \ref{B2mixCHHsym} (and trivially for an independent set).

Now we show that  the list is complete. 
So let $G$ be a finite $\CHH$ graph, and we aim to show that it is included in one of the families in the list. By Theorem~\ref{CHHcfg} we may assume that $G$ is disconnected, say $G = \underset{i \in [k]} \bigcup G_i$ where each component $G_i$ is itself connected. By Proposition~\ref{discXYrs}, each $G_i$ is $\CHH$, so it is in one of the families in Theorem~\ref{CHHcfg}.

If some component of $G$ is a copy of $K_1$, then clearly all other components must also be copies of $K_1$  (as noted in comments at the begining of this section). So $G$ must be an independent set; that is, $G$ is in family (a).

Otherwise, all components of $G$ are nontrivial.
If some component is non-bipartite, then by Theorem~\ref{CHHcfg} it must be a $K_n$-treelike graph with $n \ge 3$. Then by Lemma~\ref{treelikeCHHsym}, all components are $K_n$-treelike graphs for fixed $n \ge 3$, so $G$ is in family (b). 

Now we are in the case that all components are bipartite. 
First suppose some component of $G$ is a bipartite complement of a perfect matching, say $G_1 = \overline{L(K_{2,n})}$ with $n \ge 3$. 
Then observe that $G_1$ satisfies (B2), but does not satisfy (B1) (since it embeds $C_6$) nor (B2*). 
Then by Lemma~\ref{B1B2CHHsym}, each $G_i$ satisfies (B2). 
By Lemma~\ref{bcpmCHHsym}, any other components which are bipartite complements of perfect matchings must be isomorphic to $G_1$. Say $G_i = \overline{L(K_{2,n})}$ for each $i \in [l]$ for some $l \in [k]$.
By Lemma~\ref{B2graphs}, all other components must satisfy (B2*). 
So consider such a component $G_i$ for some $i$ such that $l < i \le k$. 
If $\Delta(G_i) > \Delta(G_1) = n-1$, then by Lemma~\ref{B2mixCHHsym}, $G_i$ is $\PCM(n)$-free. 
Otherwise, if $\Delta(G_i) \le \Delta(G_1) = n-1$, then certainly $G_i$ is $\PCM(n)$-free (since observe that if $H$ embeds a $\PCM(n)$ graph, then $\Delta(H) > n-1$).
So for $i \in \{l+1, \ldots, k\}$, each part of $G_i$ has a common neighbour, but $G_i$ is $\PCM(n)$-free.
Thus $G$ is in family (e).

Finally suppose all components of $G$ are bipartite, but no component is a bipartite complement of a perfect matching. 
If all components of $G$ satisfy (B1), that is, all induced cycles are squares and $\ts$ does not embed, then $G$ is in family (c). 
Otherwise, some component of $G$ embeds $C_6$ or $\ts$, so by Lemma~\ref{B1B2CHHsym}, all components of $G$ must satisfy (B2). But no component is a bipartite complement of a perfect matching, so they must all satisfy (B2*); that is, each component is a bipartite graph such that each part has a common neighbour. So $G$ is in family (d). 
\end{proof}

\section{The classes $\CIH$ and $\CMH$}   \label{secCIHCMH}

It remains an open problem to obtain full classifications of the finite $\CIH$ and $\CMH$ graphs. 
Nevertheless, it does not take too much more work to see that the hierarchy picture of the classes of finite connected-homomorphism-homogeneous graphs does not reduce, since there are no further inclusions. 

In fact, the finite $\IH$ graphs are not yet classified, and this is clearly a subclass of the class of finite $\CIH$ graphs. A large family of finite $\IH$ graphs is described in \cite{dcl:thesis}, called the generalised multiclaws, and we will refer to this family as (GMC). 
A \emph{generalised multiclaw} is a graph $K_m ~\overline{\cup} \left( \underset{1 \le \alpha \le l}{\overline{\bigcup}} ~(j_\alpha \cdot K_k) \right)$, 
where $j_\alpha, k, l, m \in \N \cup \{ 0 \}$ with $j_\alpha \ge 2$ for each $\alpha \in [l]$, and $k \ge 1$ (recall from Section~\ref{secChomhom} that $G ~\overline{\cup} ~H$ denotes the edge-complete union of $G$ and $H$). 
We will not go into more detail about these graphs here, except to note that the family of complete multipartite graphs is a subfamily of the family (GMC); put $k=1, m=0$. 

By the classifications in Theorems~\ref{CIIfg}, \ref{CHIfg}, \ref{CMIfg}, \ref{CHHcfg}, it is known which finite connected $\CII$ graphs are $\CHI$, $\CMI$, or $\CHH$; and we now determine which others are $\CMH$. 
First of all, we may see that line graphs of complete bipartite graphs, the Petersen graph, and the Clebsch graph are not $\CMH$.

\begin{lem} \label{LKssPCnotCMH}		
The line graph of a complete bipartite graph $L(K_{s,s})$ (with $s > 2$), the Petersen graph, and the Clebsch graph, are not $\CMH$.
\end{lem}

\begin{proof}
In each case, we find a monomorphism between connected subgraphs which cannot be extended to an endomorphism. 

Let $L(K_{s,s})$ be labelled as described in Section~\ref{secintro}, and consider the monomorphism $\phi_1$ between connected subgraphs which fixes $a_1,  a_2, b_3$ and maps $b_2$ to $a_3$. Now $b_1 \sim \{a_1, b_2, b_3 \}$, but $\phi_1( \{a_1, b_2, b_3 \} ) = \{a_1, a_3, b_3 \}$ and there is no vertex adjacent to $a_1, a_3$ and $b_3$. 

The Petersen graph, $\overline{L(K_5)}$, is obtained by letting the vertices be the 2-sets of $\{1,2,3,4,5\}$, and vertices are joined by an edge if and only if the corresponding 2-sets are disjoint. Consider the monomorphism $\phi_2$ between connected subgraphs which fixes the vertices $12,35,24,15,13$, and maps $23$ to $34$. Now $45 \sim \{ 12, 13, 23 \}$, but $\phi_2( \{ 12, 13, 23 \} ) = \{ 12, 13, 34 \}$, and there is no vertex adjacent to $12, 13$ and  $34$. 

Note that the Petersen graph is a subgraph of the Clebsch graph, $\square_5$, (induced on the non-neighbours of a vertex). 
Now we can easily find a monomorphism between connected subgraphs of $\square_5$ that corresponds to $\phi_2$ as defined above. Again this monomorphism does not extend, as there are no new vertices of $\square_5$ adjacent to the vertices required. 
\end{proof}

Finally consider the family of complete multipartite graphs. Complete graphs may be considered to be trivial complete multipartite graphs, such that all parts have size one; but we would rather discount this case, so we only consider nontrivial complete multipartite graphs. These are known to be $\CIH$ (since they are a subset of the family of generalised multiclaws), but only regular complete multipartite graphs are $\CII$ (those with parts all of the same size), and only complete bipartite graphs are $\CHH$ (those with only two parts). We now see that again only complete bipartite graphs are $\CMH$. 

\begin{lem} \label{KmnCMH}  
If $G$ is a $\CMH$ nontrivial complete $t$-partite graph, then $t = 2$.
\end{lem}

\begin{proof}
First suppose that $G = K_{s_1, \ldots, s_t}$, a nontrivial complete $t$-partite graph with parts of size $s_1, \ldots, s_t$ (with $t \ge 2$, and without loss of generality assuming $s_1 >1$), is $\CMH$. 
Assume for a contradiction that $t>2$, and consider $t$ vertices $v_1, \ldots, v_t$ each in different parts of the $t$-partition, with $v_1$ in the part $X_1$ of size $s_1$. 
Take $u_1 \in X_1, u_1 \ne v_1$, and consider the monomorphism $\phi$ which fixes $v_i$ for $i \in [t-1]$, and maps $u_1$ to $v_t$ (this is a map between connected subgraphs since $t>2$).
Now $v_t \sim \{ v_1, \ldots, v_{t-1}, u_1 \}$, but $ \phi( \{ v_1, \ldots, v_{t-1}, u_1 \} ) = \{ v_1, \ldots, v_{t-1}, v_t \} $, and there is no vertex adjacent to $\{ v_1, \ldots, v_t \}$ since it is a maximal clique of $G$.
So $\phi$ cannot be extended to an edge-preserving mapping which maps $v_t$, and hence $\phi$ cannot be extended to a homomorphism from $G$ to $G$; so we have a contradiction.
\end{proof}

We now aim to bring together all of the information that we have about the classes of finite connected $\C$-homomorphism-homogeneous graphs. 
The following picture, Figure~\ref{Chomhomfgpic}, shows the finite connected graphs known to be contained in each of the classes. In the picture, each class is represented by an ellipse (these are labelled underlined inside the ellipse). The ellipses subdivide the picture into many regions, determined by the way the classes are contained within one another, and how they may intersect. Families of graphs, and a few special individual graphs, are shown in the corresponding regions. 

\begin{figure}[h!tb]
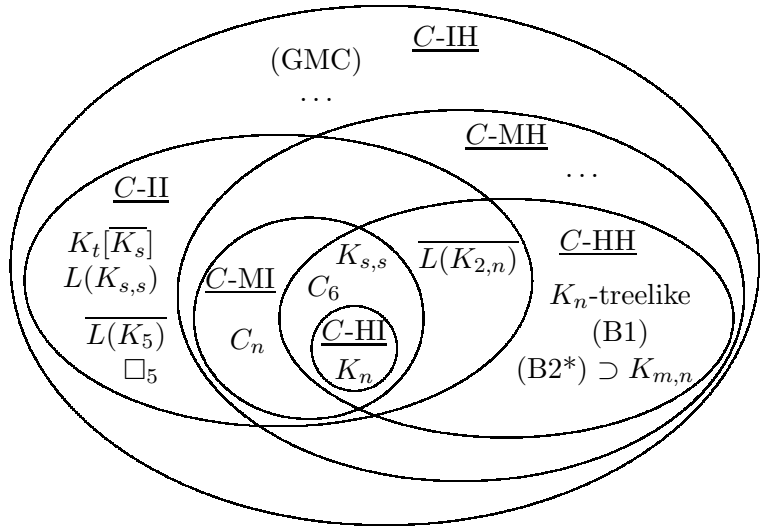
		
\hspace{2.5cm}
\xy
(52,52)*{\underline{\CIH}} ,
(35,49)*{\textrm{(GMC)}} ,
(35,44)*{\ldots} ,
(12,32)*{\underline{\CII}} ,
(8,25)*{K_t[\overline{K_s}]} ,
(8,20)*{L(K_{s,s})} ,
(10,13)*{\overline{L(K_5)}} ,
(12,8)*{\square_5} ,
(25,20)*{\underline{\CMI}} ,
(26,12)*{C_n} ,
(60,39)*{\underline{\CMH}} ,
(70,34)*{\ldots} ,
(41,23)*{K_{s,s}} ,
(36,19)*{C_6} ,
(40,13)*{\underline{\CHI}} ,
(40,8)*{K_n} ,
(55,23)*{\overline{L(K_{2,n})}} ,
(72,25)*{\underline{\CHH}} ,
(75,18)*{K_n \textrm{-treelike}} ,
(75,13)*{\textrm{(B1)}} ,
(73,8)*{\textrm{(B2*)} \supset K_{m,n}} ,
(40,11)*\xycircle<16pt,16pt>{} , 
(34,15)*\xycircle<43pt,38pt>{} , 
(30,20)*\xycircle<95pt,55pt>{} , 
(60,15)*\xycircle<85pt,45pt>{} , 
(54,18)*\xycircle<106pt,70pt>{} , 
(44,22)*\xycircle<140pt,98pt>{} , 
\endxy
\caption{The classes of finite connected $\C$-homomorphism-homogeneous graphs.} \label{Chomhomfgpic}
\end{figure}

The family of generalised multiclaw graphs is denoted by (GMC), while the families of graphs which satisfy the properties (B1) and (B2*) are represented using these labels. 
Other families of graphs are represented by a typical element from the family. For example, `$K_{s,s}$' is used to denote the family of regular complete bipartite graphs $\{ K_{s,s}: s \in \N \}$. Variables within families may generally be assumed to range over all natural numbers, but strictly speaking  note the lower limits in the following specific cases:
$K_{s,s}$, $s \ge 2$; 
$\overline{L(K_{2,n})}$, $n \ge 3$; 
$K_n$-treelike, $n \ge 2$; 
$C_n$, $n \ge 3$; 
$K_t[\overline{K_s}]$, $s,t \ge 2$; 
$L(K_{s,s})$, $s \ge 3$. 

Naturally, families of graphs are shown in the picture located inside the region of a class in which the whole family is found, but note that some subfamilies or special cases may be contained in other subclasses. For example, the family of trivial $K_n$-treelike graphs is precisely the family of complete graphs, $K_n$; a typical $K_n$-treelike graph is $\CHH$ (but not $\CHI$), while complete graphs are in fact $\CHI \subset \CHH$. Similarly, the family $K_{s,s}$ is a subfamily of the family $K_t[\overline{K_s}]$ such that $t=2$, and a subfamily of the family $K_{m,n}$ such that $m=n$; a typical $K_t[\overline{K_s}]$ graph is $\CII$ (but not $\CMH$), a typical $K_{m,n}$ graph is $\CHH$ (but not $\CII$) while a typical $K_{s,s}$ graph is in fact $\CMI$ (but not $\CHI$). 
Note that $C_3 = K_3$ which is $\CHI$; $C_4 = K_{2,2}$ and $C_6 = \overline{L(K_{2,3})}$ which are both $\CMI$ and $\CHH$ but not $\CHI$.

Finally, note that since the classes $\CII$ and $\CHH$ are completely classified, and we have verified which finite connected $\CII$ graphs are $\CMH$, this part of the picture is completely determined (that is, no other graphs lie in the regions within the ellipses representing these classes). So for instance, the region $\CII \cap \CMH \setminus (\CMI \cup \CHH)$ really is empty; that is, any finite connected graph which is $\CII$ and $\CHH$ must in fact be $\CMI$ or $\CHH$. In contrast, the classes $\CMH$ and $\CIH$ are not classified, so there may be other graphs lying in these classes (in fact this is very likely); this is represented by dots `$\dots$' in the corresponding regions.

It is clear to see from this picture that the hierarchy picture of the classes of finite connected-homomorphism-homogeneous graphs does not reduce, since the classes are clearly all distinct, and there are no further inclusions. As stated at the beginning of this section, it remains an open problem to obtain classifications for the finite connected $\CMH$ and $\CIH$ graphs. In fact it would be interesting to even know of other examples of finite connected graphs which are $\CMH$ but not $\CHH$ or $\CMI$; and $\CIH$ but not $\CMH$ or $\CII$ or $\IH$. Other possibilies for extensions would be to look at classifying countable (not just finite) $\CHH$ graphs; or obtaining classifications of the connected-homomophism-homogeneous classes for other types of relational structures (finite digraphs and countable posets immediately present themselves as candidates given the work mentioned in Section~\ref{secintro}). 

\section*{Acknowledgements}

This work was supported by EPSRC grant EP/H00677X/1.

\end{document}